\documentclass[moor]{informs3} 

\usepackage{endnotes}
\let\enotesize=\normalsize
\def\notesname{Endnotes}%
\def\makeenmark{$^{\theenmark}$}
\def\enoteformat{\rightskip0pt\leftskip0pt\parindent=1.75em
  \leavevmode\llap{\theenmark.\enskip}}

\usepackage[numbers]{natbib}
\NatBibNumeric
 \def\bibfont{\small}%
\bibpunct[, ]{[}{]}{,}{n}{}{,}%

\TheoremsNumberedThrough    
\EquationsNumberedThrough    
                 
\usepackage{amssymb,amsmath}
\renewenvironment{proof}{{\textit{Proof.}}}{\hfill$\square$}
\usepackage{enumitem,multirow,multicol}
\usepackage{microtype}
\usepackage{caption,subcaption}
\usepackage{algorithmic} 
\usepackage[ruled,vlined]{algorithm2e}
\usepackage{booktabs}
\usepackage[colorlinks=true,breaklinks=true,bookmarks=true,urlcolor=blue,citecolor=blue,linkcolor=blue,bookmarksopen=false,draft=false]{hyperref}

\usepackage{changepage}
\usepackage{graphicx}

\def\EMAIL#1{\href{mailto:#1}{#1}}

\EquationsNumberedThrough

\newcommand{\ww}{\mathbf{w}}

\begin{document}

\RUNAUTHOR{Kim}
\RUNTITLE{A Lagrangian with False Penalty for Nonconvex Optimization}

\TITLE{
A Lagrangian-Based Method with ``False Penalty" for Linearly Constrained Nonconvex Composite Optimization
}
\ARTICLEAUTHORS{%
\AUTHOR{Jong Gwang Kim}
\AFF{School of Industrial Engineering, Purdue University, West Lafayette, IN 47906, \EMAIL{kim2133@purdue.edu}}
} 

\ABSTRACT{
We introduce a primal-dual framework for solving linearly constrained nonconvex composite optimization problems. Our approach is based on a newly developed Lagrangian, which incorporates \emph{false penalty} and dual smoothing terms. This new Lagrangian enables us to develop a simple first-order algorithm that converges to a stationary solution under standard assumptions. We further establish global convergence, provided that the objective function satisfies the Kurdyka-Łojasiewicz property. Our method provides several advantages: it simplifies the treatment of constraints by effectively bounding the multipliers without boundedness assumptions on the dual iterates; it guarantees global convergence without requiring the surjectivity assumption on the linear operator; and it is a single-loop algorithm that does not involve solving penalty subproblems, achieving an iteration complexity of $\mathcal{O}(1/\epsilon^2)$ to find an $\epsilon$-stationary solution.  Preliminary experiments on test problems demonstrate the practical efficiency and robustness of our method.
}

\maketitle

\section{Introduction} \label{intro}
We consider the nonconvex optimization with linear constraints:
\begin{equation} \label{eq:op}
    \underset{x \in \mathbb{R}^n}{\mathrm{min}} \ \ f(x) + h(x) \ \ \mathrm{s.t.} \ \  Ax = b,
\end{equation}
where $f: \mathbb{R}^n \rightarrow \mathbb{R}$ is a continuously differentiable (possibly nonconvex) function with $L_f$-Lipschitz gradient; $h: \mathbb{R}^n \rightarrow \mathbb{R} \cup \{+\infty\}$ is a proper closed convex (not necessarily smooth) function; and $A:  \mathbb{R}^m \rightarrow \mathbb{R}^n$ is a linear operator and $b \in \mathbb{R}^m$. 

Solving nonconvex problems, even those without constraints, is generally challenging, and it is often computationally intractable to find even an approximate global minimum  (\citet{nemirovskij1983problem}). Furthermore, problem \eqref{eq:op} frequently arises in a variety of applications and tends to be large-scale (\citet{boyd2011distributed}). In this paper, our objective is to provide an efficient first-order method with theoretical guarantees for computing a stationary solution to problem \eqref{eq:op}. In particular, we present a single-loop first-order method, based on a new Lagrangian, to find an $\epsilon$-stationary solution (Definition \ref{def_epsilon-kkt}). We show that our method achieves an iteration complexity of $\mathcal{O}(1/\epsilon^{2})$ and ensures global convergence.

Our approach is closely related to the augmented Lagrangian (AL) framework, introduced by \citet{hestenes1969multiplier} and \citet{powell1969method}. This framework has been a powerful algorithmic framework for constrained optimization, including problem \eqref{eq:op} (see \citet{bertsekas2014constrained,birgin2014practical}, and references therein). In recent years, there has been renewed interest in AL-based methods, especially within the context of the Alternating Direction Multiplier Method (ADMM) scheme (\citet{glowinski1975approximation}). This renewed interest is largely due to the beneficial properties of AL-based methods, such as scalability and excellent practical performance in solving large-scale problems that arise in data science and machine learning; see, e.g., \citet{boyd2011distributed,latorre2019fast,scheinberg2010sparse,yang2017alternating} and references therein. 

The convergence and iteration complexity of AL-based methods for convex problems have been extensively studied and well established in the literature (see e.g., \citet{aybat2012first,lan2016iteration,liu2019nonergodic,ouyang2015accelerated,patrascu2017adaptive,shefi2014rate,xu2017accelerated,xu2021iteration}, among others). Given that the literature on AL-based methods is quite vast, we focus our review on the literature dealing with the iteration complexity and global convergence of AL-based methods for solving linearly constrained nonconvex problems.

\subsection{Related Work}
Recent research has focused on the iteration complexity of first-order AL-based methods for solving the nonconvex problem \eqref{eq:op}. Several notable approaches have been proposed in the literature. \citet{hajinezhad2019perturbed} proposed a perturbed-proximal primal-dual algorithm that converges to a first-order stationary solution under the assumption of initialization feasibility. This algorithm obtains an iteration complexity of $\mathcal{O}(1/\epsilon^{4})$. \citet{kong2019complexity} proposed a penalty method that utilizes an inner accelerated composite gradient to solve subproblems, achieving a complexity result of $\mathcal{O}(1/\epsilon^{3})$. Building on this work, \citet{kong2023iteration} further improved the complexity to $\mathcal{O}(1/\epsilon^{2.5})$ under Slater's condition by incorporating an accelerated composite gradient into the proximal AL methods. However, it is important to note that these methods require double-loops, which can increase the computational workload. In a different approach, \citet{zhang2020proximal,zhang2022global} presented a single-loop proximal AL method (SProx-ALM) for linearly constrained problems with a box constraint set \cite{zhang2020proximal} or a polyhedral set \cite{zhang2022global}. The authors showed that SProx-ALM is an {\em order-optimal} algorithm that achieves $\mathcal{O}(1/\epsilon^{2})$ iteration complexity with a hidden constant that depends on Hoffman’s constraints.

Another important line of research focuses on global convergence in the context of linearly constrained nonconvex optimization problems. Recent advances in AL-based algorithms have provided global convergence guarantees for these problems; see, e.g., \citet{bot2019proximal,boct2020proximal,li2015global,wang2019global,yang2017alternating,zeng2022moreau}. These algorithms do not impose any boundedness assumptions on the dual iterates, but rely on the assumption that every linear operator is surjective (i.e., full row rank matrix $A$) to ensure global convergence to a stationary solution. In a related important development, \citet{bolte2018nonconvex} and \citet{hallak2023adaptive} provided general AL-based frameworks with global convergence guarantees for nonconvex nonsmooth optimization problems with general constraints, including linear constraints.

\subsection{Contributions}
This paper makes the following contributions to the literature: 

\begin{itemize}

\item We introduce a new Lagrangian combined with artificial variables, which we call \emph{Proximal-Perturbed Lagrangian}. The artificial variables are used to get rid of the constraints with {\em false penalty} and dual smoothing (proximal) terms are added, leading to the strong concavity of the Lagrangian in the multipliers. Based on the new Lagrangian, we develop a single-loop first-order algorithm that guarantees convergence to a stationary solution. Our algorithm obtains an $\epsilon$-stationary solution with an iteration complexity of $\mathcal{O}(1/\epsilon^{2})$, which matches the best known $\mathcal{O}(1/\epsilon^{2})$ complexity of the algorithm in \citet{zhang2020proximal,zhang2022global}. 

\item We provide a relatively simple proof procedure to establish the complexity bound of $\mathcal{O}(1/\epsilon^{2})$ and global convergence under standard assumptions. Importantly, the structure of our proposed algorithm allows us to leverage the innovative proof technique proposed by \citet{gur2023convergent} for the unconstrained nonconvex setting, and adapt it effectively for our constrained setting. In addition, we do not impose the boundedness assumption on the multiplier sequence (\citet{bolte2018nonconvex,hallak2023adaptive}) nor the subjectivity of the linear operator $A$. Furthermore, our algorithm does not require the feasibility of initialization, the strict complementarity condition, and Slater's condition.

\item Our method has a practical advantage over other AL-based methods due to the use of fixed (false) penalty parameter. This feature simplifies the implementation of the algorithm by removing computational efforts in tuning penalty parameters and sensitivity to the choice of penalty parameters. Numerical results show that the fixed parameter, along with the bounded multipliers, leads to a consistent reduction in both first-order optimality and feasibility gaps.
\end{itemize}

\subsection{Outline of the paper}
The paper is organized as follows. Section \ref{prelim} provides the notation, definitions, and assumptions that we will use throughout the paper. In Section \ref{ppl}, we introduce the new Lagrangian function and propose a first-order primal-dual algorithm. In Section \ref{analysis}, we establish the convergence results of our algorithm. Section \ref{experiment} presents preliminary numerical results to demonstrate the effectiveness of the proposed algorithm.

\section{Preliminaries} \label{prelim}
This section provides the notation, definitions, and assumptions we will use throughout the paper. 

We let $\mathbb{R}^n$  denote the $n$-dimensional Euclidean space with inner product $\left\langle x, y \right\rangle$ for $x,y \in \mathbb{R}^n$. The Euclidean norm of a vector is denoted by $\left\| \: \cdot \: \right\|$. The distance function between a vector $x$ and a set $X \subseteq \mathbb{R}^n$ is defined as $\text{dist}(x,X) := \inf_{y \in X} \|y-x\|.$
For the matrix $A \in \mathbb{R}^{m \times n}$, the largest singular value of $A$ is denoted by $\sigma_{\max}$.
For a proper closed convex function $g: \mathbb{R}^n \rightarrow \mathbb{R} \cup \{ +\infty\}$, the domain of $g$ is defined as 
$\mathrm{dom}(h) := \{x \in \mathbb{R}^n \ | \ h(x) < +\infty\}.$ 
The function is said to be proper if $\text{dom}(h) \neq \emptyset$ and does not take the value $-\infty$. The function is called {\em closed} if it is lower semicontinuous, i.e., $\liminf_{x \rightarrow x^0}  h(x) \geq h(x^0)$ for any point $x^0 \in \mathbb{R}^n$. For any set $X \subseteq \mathbb{R}^n$, its indicator function $\mathcal{I}_X$ is defined by
$\mathcal{I}_X = 0   \text{ if }   x \in X \ \text{ and }  + \infty, \text{ otherwise}.$
We denote the subdifferential of a convex function $g$ at a point $x$ by $\partial h(x)$ (\citet[Definition 8.3]{rockafellar2009variational}):
\[
\partial h(x) := \left\lbrace v \in \mathbb{R}^n: h(y) \geq h(x) + \left\langle v, y-x \right\rangle  \ \ \forall y \in \mathbb{R}^n \ \text{and} \ x \in \text{dom}(h) \right\rbrace.
\]
Given $x \in \mathbb{R}^n$ and $\eta>0$, the {\em proximal map} associated with $g$ is defined by 
\begin{equation}
    \text{prox}_{\eta g} (x) := \underset{y \in \mathbb{R}^n}{\text{argmin}}\left\lbrace h(x) + \frac{1}{2\eta} \| x - y\|^2 \right\rbrace. \notag
\end{equation}

The stationary solutions of problem \eqref{eq:op} can be characterized by the points $(x^\ast, \lambda^\ast)$ that satisfy the Karush-Kuhn-Tucker (KKT) conditions:
\begin{definition}[KKT solution] \label{def_kkt}
We say a point $x^\ast$ is a \emph{KKT solution} for problem \eqref{eq:op} if there exists $\lambda^\ast \in \mathbb{R}^m$ such that
\begin{equation} \label{eq:def_kkt}
	0 \in \nabla f(x^\ast) + \partial h(x^\ast) +  A^{\top}\lambda^\ast, \quad Ax^\ast -b = 0.
\end{equation}
\end{definition}

We define $\epsilon$-KKT point (or $\epsilon$ -stationary solution) of the problem \eqref{eq:op}.

\begin{definition}[$\epsilon$-KKT solution] \label{def_epsilon-kkt}
Given $\epsilon>0$, a point $x^\star$ is said to be an {\em $\epsilon$-KKT solution} for problem \eqref{eq:op} if there exists $\lambda^\star \in \mathbb{R}^m$ such that
\begin{equation} \label{eq:def_epsilon-kkt}
\mathrm{dist}(0, \nabla f(x^\star) + \partial h(x^\star) + A^{\top}\lambda^\star) \leq \epsilon, 
	\quad \| Ax^\star-b \| \leq \epsilon,
\end{equation}
where $\partial h(x^\star)$ is the general subdifferenctal of $g$ at $x^\star$.
\end{definition}

Throughout the paper, we make the following assumptions on problem \eqref{eq:op}.

\begin{assumption}[Existence of KKT solution] \label{assumption_kkt}
There exists a primal-dual solution $(x, \lambda) \in \text{dom}(h) \times \mathbb{R}^m$ that satisfies the KKT conditions \eqref{eq:def_kkt}.
\end{assumption}

\begin{assumption}[Smoothness]\label{assumption_lipschitz}
Given the domain $\text{dom}(h) \subseteq \mathbb{R}^n $, $\nabla f$ is  $L_f$-Lipschitz continuous, i.e., there exists a constant $L_f > 0$ such that 
\begin{align}
    \|\nabla f(x) -\nabla f(x^\prime)\| & \leq L_f \|x-x^\prime\|, \quad \forall x,x^\prime \in \mathrm{dom}(h). \label{eq:lipschitz_i} 
\end{align}	
\end{assumption}

\begin{assumption}[Bounded domain] \label{assumption_bounded_domain}
The domain of the function $h$ is compact, i.e.,
\[
\underset{x,x^\prime \in \mathrm{dom}(h)}{\max} \|x-x^\prime\| < + \infty.
\]
\end{assumption}

The assumptions above are quite standard and are satisfied by a wide range of practical problems. Note that we do not impose some restrictive assumptions, including the feasibility of initialization (\citet{hajinezhad2019perturbed}), the strict complementarity condition (\citet{zhang2020proximal}), and Slater's condition (\citet{kong2019complexity,zhang2020proximal}). Moreover, we do not make the assumption of the full rank of the matrix $A$ (\citet{li2015global,boct2020proximal}).

\section{Proximal-Perturbed Lagrangian Method} \label{ppl}
In this section, we develop a first-order algorithm based on a new Lagrangian and observe some of its properties. 

\subsection{Proximal-Perturbed Lagrangian}
We begin by converting problem \eqref{eq:op} into an extended formulation by introducing \emph{perturbation} variables $z \in \mathbb{R}^m$ and letting $z = 0$ and $h(x) = z$:
\begin{equation} \label{eq:ep}
\begin{aligned}
	\underset{x \in \mathbb{R}^n,\: z \in \mathbb{R}^m}
	{\mathrm{min}}  \ \ f(x) + h(x)  \ \ 
	\mathrm{s.t.}  \ \ Ax - b=z, 
	                      \ \ z = 0.
\end{aligned}
\end{equation} 
Clearly, for the unique solution $z^\ast=0$, the above formulation is equivalent to problem \eqref{eq:op}. 
Let us define the \emph{Proximal-Perturbed Lagrangian} (P-Lagrangian) for problem \eqref{eq:ep}:
\begin{equation} \label{eq:PPL} 
\mathcal{L}_{\beta}(x,z,\lambda,\mu) =   f(x) + \left\langle \lambda, Ax - b - z \right\rangle + \left\langle \mu, z \right\rangle + \frac{\alpha}{2}\| z \|^2 - \frac{\beta}{2}\| \lambda - \mu \|^2 + h(x),
\end{equation}
where $\lambda \in \mathbb{R}^m$ and $\mu \in \mathbb{R}^m$ are the Lagrange multipliers associated with the constraints $Ax-b-z=0$ and $z = 0$, respectively. Here, $\alpha>0$ is a penalty parameter and $\beta>0$ is a dual proximal parameter.

Notice that the structure of $\mathcal{L}_{\beta}(x,z,\lambda,\mu)$ differs from the standard AL function and its variants.  It is characterized by the absence of a penalty term for handling linear constraint $Ax - b - z = 0$. Only the additional constraint $z = 0$ is penalized with the quadratic term $\frac{\alpha}{2}\| z \|^2$, termed {\color{blue} {``false penalty''}}\footnote{The term ``false penalty" draws an analogy from the ``False Nine" role in soccer. A false nine is a player who, despite being positioned as a Forward (classically target scoring goals), instead retreats into midfield to help control the game and create scoring opportunities (e.g., Messi). Similarly, the false penalty, combined with dual smoothing, guides the algorithm towards satisfying the constraints rather than directly penalizing constraint violations; see subsection \ref{subsec:false_penalty} for details.}, while $Ax - b - z = 0$ is relaxed into the objective with the corresponding multiplier. Additionally, the negative quadratic term $-\frac{\beta}{2}\| \lambda - \mu\|^2$, termed {\em dual smoothing}, makes $\mathcal{L}_{\beta}$ smooth and strongly concave in $\lambda$ for fixed $\mu$ and in $\mu$ for fixed $\lambda$. Note that due to the strong convexity of $\mathcal{L}_{\beta}(x,z,\lambda,\mu)$ in $z$, there exists a unique solution for given $(\lambda,\mu)$. If we minimize $\mathcal{L}_{\beta}(x,z,\lambda,\mu)$ with respect to $z$, we have
\begin{equation} \label{eq:opt_z}
    z(\lambda,\mu)={(\lambda - \mu)}/{\alpha}, 
\end{equation}
which implies  $\lambda=\mu$ at the unique solution $z^{\ast}=0$.
Based on this relation on $\lambda$ and $\mu$ at $z^\ast = 0$, we have added the term $-\frac{\beta}{2} \|  \lambda-\mu \|^2$ to the Lagrangian in \eqref{eq:PPL}.
Plugging $z(\lambda, \mu)$ into $\mathcal{L}_{\beta}(x,z,\lambda,\mu)$ yields the reduced P-Lagrangian: 
\begin{align} \label{eq:reduced_ppl}
	\mathcal{L}_{\beta}(x,z(\lambda, \mu),\lambda,\mu) 
	 = f(x) + \left\langle \lambda, Ax-b \right\rangle
	- \frac{1}{2\rho} \| \lambda- \mu \|^2 + h(x). 
\end{align}
Since $\mathcal{L}_{\beta}(x,z(\lambda, \mu),\lambda,\mu)$ is strongly concave in $\lambda$ for given $(x,\mu)$, there exists a unique maximizer, denoted by $\lambda(x,\mu)$. Maximizing the reduced P-Lagrangian in \eqref{eq:reduced_ppl} with respect to $\lambda$, we obtain
\begin{equation} \label{eq:opt_lambda}
	\lambda(x,\mu) = \underset{\lambda \in \mathbb{R}^m}{\textrm{argmax}} \; \mathcal{L}_{\beta}(x,z(\lambda,\mu),\lambda,\mu)
	= \mu + \rho (Ax-b), 
\end{equation}
from which we derive the $\lambda$-update step \eqref{eq:lambda_update}.

\subsection{Algorithm}
We present a first-order algorithm that utilizes the features of the P-Lagrangian to compute a stationary solution of problem \eqref{eq:op}. The steps of the proposed algorithm are outlined in Algorithm \ref{algorithm1}.

\begin{algorithm}
\caption{P-Lagrangian-Based First-Order Primal-Dual Algorithm.}  \label{algorithm1}

{{\bf Input:} $\alpha \gg 1$, $\beta \in (0,1)$, $\rho:=\frac{\alpha}{1+\alpha \beta}$, and $r \in (0.9,1)$, and $0 < \eta < \frac{1}{L_f +  \left(2 + \frac{1}{1+ \alpha \beta}\right) \rho \sigma_{\max}^2}$.}

{{\bf Initialization:} $(x_0,z_0,\lambda_0,\mu_0) \in \mathbb{R}^n \times \mathbb{R}^m \times \mathbb{R}^m \times \mathbb{R}^m$ and $\delta_0 \in (0,1]$.}
\vspace{0.025in}

\For {$k =0,1,2,\ldots$} 
{\begin{flalign}
& x_{k+1} = \underset{x \in \mathbb{R}^n}{\mathrm{argmin}} \left\lbrace  \left\langle \nabla_x \ell_{\beta}(x_k,z_k,\lambda_k,\mu_k), x - x_k \right\rangle + \frac{1}{2\eta} \| x - x_k \|^2 + h(x) \right\rbrace; & \label{eq:x_update}\\ 
& \mu_{k+1} = \mu_k + \tau_k (\lambda_k - \mu_k) \ \ \text{with} \ \ \tau_k =  \frac{\delta_k}{\| \lambda_k - \mu_k \|^2 + 1};  & \label{eq:mu_update} \\
& \lambda_{k+1} = \mu_{k+1} + \rho (Ax_{k+1} -b); & \label{eq:lambda_update} \\
& z_{k+1} = \frac{\lambda_{k+1} - \mu_{k+1}}{\alpha}; & \label{eq:z_update} \\
& \delta_{k+1} = r \delta_k.
\end{flalign}
}
\end{algorithm}  	
The exact minimization of $\mathcal{L}_{\beta}$ with respect to $x$ is challenging due to the nonconvexity of $f$. To address this, we employ an approximation $\widehat{\mathcal{L}}_{\beta}$ in $x$ at a point $y$ (see e.g., \citet{bolte2014proximal}):
\begin{equation} \label{eq:approx_PPL}
\widehat{\mathcal{L}}_{\beta} (x,z,\lambda,\mu;y) := \ell_{\beta}(y,z,\lambda,\mu)
+\left\langle \nabla _x \ell_{\beta}(y,z,\lambda,\mu), x - y \right\rangle  
+\frac{1}{2\eta}\| x - y \|^2 + h(x), 
\end{equation}
where $\ell_{\beta}$ represents the smooth part of $\mathcal{L}_{\beta}$:
\begin{equation} \label{eq:smooth_part_PPL}
	\ell_{\beta}(x, z, \lambda, \mu) :=  f(x) + \left\langle \lambda, Ax - b - z \right\rangle + \left\langle \mu, z \right\rangle + \frac{\alpha}{2}\| z \|^2 - \frac{\beta}{2}\| \lambda - \mu \|^2, \notag
\end{equation}
This approximation is so-called the proximal linearized approximation of $\mathcal{L}_{\beta}$ in $x$. 
Note that we can adopt alternative approximations for $\widehat{\mathcal{L}}_{\beta}$, depending on the problem data (see e.g., \citet{razaviyayn2013unified,scutari2016parallel,scutari2014decomposition}).

The first step of the algorithm is update $x$ by performing a minimization of  $\widehat{\mathcal{L}}_{\beta} (x,z_k,\lambda_k,\mu_k;x_k)$ in $x$ while keeping   $(z_k,\lambda_k,\mu_k)$ fixed:
\begin{equation} 
	x_{k+1} = \underset{x\in \mathbb{R}^n}{\mathrm{argmin}} \left\lbrace \left\langle \nabla_x \ell_{\beta}(x_k,z_k,\lambda_k,\mu_k), x-x_k \right\rangle  
	+\frac{1}{2\eta} \| x - x_k \|^2 + h(x)\right\rbrace, \notag
\end{equation}
which is known as the {\em proximal gradient map} and it can be rewritten as
\begin{equation}
	x_{k+1} = \text{prox}_{\eta h} \left[ x_k - \eta \nabla_x \ell_{\beta}(x_k,z_k,\lambda_k,\mu_k)\right]. \notag
\end{equation}

Next, the {\em auxiliary} multiplier $\mu$ is updated as follows:
\begin{equation} 
	\mu_{k+1} = \mu_k + \tau_k (\lambda_k-\mu_k). \notag
\end{equation}
Here, the step size $\tau_k$ defined by
$$\tau_k =  \frac{ \delta_k}{\| \lambda_k - \mu_k \|^2 + 1},$$
where $\delta_k = r^k \delta_0 $ and $r \in (0.9,1)$. It is important to note that $\delta_k =  r^k \delta_0 $ is a summable sequence, i.e., $\sum_{k=0}^{+ \infty} \delta_k < + \infty$. The key benefit of this choice of $\tau_k$ is that it guarantees that the multiplier sequence $\{\mu_k\}$ is bounded, which in turn ensures the boundedness of $\{ \lambda_k \}$ (see Lemma \ref{lem_mu_bound} below).

Then, for given $(x_{k+1}, \mu_{k+1})$, the multiplier $\lambda$ is updated by using \eqref{eq:opt_lambda}:
\begin{align} 
\lambda_{k+1} = \underset{\lambda\in\mathbb{R}^m}{\mathrm{argmax}}\left\lbrace f(x_{k+1}) + \left\langle \lambda, Ax_{k+1} - b \right\rangle - \frac{1}{2\rho}\| \lambda- \mu_{k+1} \|^2 + h(x_{k+1})\right\rbrace,     \notag 
\end{align}
equivalently,
\begin{equation} 
\lambda_{k+1} = \mu_{k+1} + \rho  (Ax_{k+1} -b). \notag
\end{equation}

The last step is to update $z$ using an exact minimization step on $\mathcal{L}_{\beta}$: 
\begin{equation} 
	z_{k+1} = \frac{\lambda_{k+1}-\mu_{k+1}}{\alpha}, \notag
\end{equation}
where $\alpha > 0$ is a fixed (false) penalty parameter.

\begin{lemma} \label{lem_mu_bound}
Let $\{(x_k, z_k, \lambda_k, \mu_k) \}$ be the sequence generated by Algorithm \ref{algorithm1}. Then,  the multiplier sequences $\{ \mu_k \}$ and $\{ \lambda_k \}$ are bounded. 
\end{lemma}
\begin{proof}
From the $\mu$-update step \eqref{eq:mu_update} with  $\mu_0=0$, we directly deduce
\begin{align}
\| \mu_{k+1} \| 
& = \| \mu_0 + \sum^k_{t=0} \tau_k ( \lambda_t - \mu_t ) \|  
\leq  \sum^{+\infty}_{t=0}  \frac{\delta_t}{\| \lambda_t - \mu_t \|^2 + 1} \cdot \| \lambda_t - \mu_t \| \notag \\
& \leq  \sum^{+\infty}_{t=0}  \frac{\delta_t}{\| \lambda_t - \mu_t \|  + \frac{1}{\| \lambda_t - \mu_t \|}} 
\leq  \frac{1}{2}\sum^{ + \infty}_{t=0} \delta_t, \notag
\end{align} 
where in the last inequality, we used the fact that $a + b \geq 2\sqrt{ab}$ for any $a,b \geq 0$. Note that $\sum^{\infty}_{t=0} {\delta_t}$ is convergent with $\delta_t = r^t \delta_0$ and $r \in (0.9, 1)$. Hence, $\{ \mu_k \}$ is bounded. Given the update of $\lambda_{k+1} = \mu_{k+1} + \rho (A x_{k+1} -b)$, where $ \{A x_{k+1} - b\}$ is bounded over $\text{dom}(h)$ (Assumption \ref{assumption_bounded_domain}) and $\rho = \frac{\alpha}{1+\alpha \beta}$ is a constant, the sequence $\{\lambda_k\}$ is also bounded. 
\end{proof}

\begin{remark} \label{remark_step_mu}
 When updating the multiplier $\mu_k$, it is important to choose the reduction ratio $r$ close to 1 (e.g. 0.99 or even closer to 1 but less than 1). Choosing a small value of $r$ will cause the multiplier $\mu_k$ to reach a point quickly in a small number of iterations, which in turn may cause the multiplier $\lambda_k$ to stay far away from the multiplier $\lambda$ satisfying the KKT conditions \eqref{eq:def_kkt}. 
\end{remark}

\subsection{False Penalization with Dual Smoothing} \label{subsec:false_penalty}

The false penalty $\frac{\alpha}{2} \| z\|^2$ does not directly penalize constraint violation, unlike a typical penalty term. Instead, it guides the iterates $z_{k}$ towards convergence, helping to reduce constraint violation. Specifically, in the $z$-update step  \eqref{eq:z_update}, a large $\alpha > 0$ is chosen, causing $\frac{\alpha}{2} \| z\|^2$ to dominate $\left\langle \lambda - \mu, z\right\rangle$. This, along with the boundedness of $(\lambda - \mu)$ (Lemma \ref{lem_mu_bound}), leads to $\| z_{k+1} \| $ tending to a value close to 0. 

This is further facilitated by the dual smoothing term $-\frac{\beta}{2} \| \lambda - \mu_{k+1} \|^2$. In the $\lambda_{k+1}$-update step \eqref{eq:lambda_update}, $\lambda_{k+1}$ maximizes $\left\langle \lambda, Ax_{k+1} - b \right\rangle - \frac{1}{2\rho}\| \lambda- \mu_{k+1} \|^2$  exactly, which involves minimizing the strongly concave term $-\frac{1}{2\rho}\| \lambda- \mu_{k+1} \|^2$. This step encourages $\lambda_{k+1}$ to approach a point close to $\mu_{k+1}$, which in turn influences the update $z_{k+1} = \frac{\lambda_{k+1} - \mu_{k+1}}{\alpha} \left(= \frac{1}{1+\alpha\beta}(Ax_{k+1} -b) \right)$. As $\lambda_{k+1}$ gets closer to $\mu_{k+1}$, this term becomes smaller, driving $z_{k+1}$ closer to 0.

Therefore, there must exist a large $\alpha > 0$ such that for any $k \geq 0$
\[\|z_{k+1}\| \leq \alpha \| z_{k+1} - z_{k} \|.
\]
This inequality enables Algorithm \ref{algorithm1} to reduce infeasibility by controlling $\{x_{k+1} - x_{k}\}$ and $\{z_{k+1} - z_{k}\}$ with a sequence of nonnegatve values $\{\delta_k\}$ that decrease to 0 and the fixed (false) penalty parameter $\alpha > 0$ (see Theorem \ref{lem_lagrangian_properties}\ref{lem_one_iter_PPL}).

\section{Convergence Analysis} \label{analysis}

In this section, we present the convergence results of Algorithm \ref{algorithm1}. The structure of Algorithm \ref{algorithm1} allows us to establish its convergence properties in a simple way. 

Recall the Lipschitz continuity of $ \nabla_x \ell_{\beta}$. Noting that 
$\nabla_x \ell_{\beta}(x,z,\lambda,\mu) = \nabla f(x) + A^{\top} \lambda$, we have
\begin{align}
\| \nabla_x \ell_{\beta}(x_{k+1}) - \nabla_x \ell_{\beta}(x_k) \| 
\leq  \| \nabla f(x_{k+1}) - \nabla f(x_k) \| 
 \leq  L_f  \| x_{k+1} - x_k \|,  \notag
\end{align}
where $L_f$ denotes a Lipschitz constant, and we omitted $(z_k,\lambda_k,\mu_k)$ for simplicity. Then it follows from the descent lemma \cite[Proposition A.24]{bertsekas1999nonlinear} that the following inequality holds:
\begin{equation}\label{eq:descent_lemma}
\ell_{\beta}(x_{k+1}) \leq\ell_{\beta}(x_k) + \left\langle \nabla_x \ell_{\beta}(x_k), x_{k+1} - x_k \right\rangle + \frac{L_f}{2} \| x_{k+1} - x_k \|^2,
\end{equation}

Let us first provide basic yet crucial relations on the sequences $\{\lambda_k\}$, $\{\mu_k\}$, and $\{x_k\}$. These relations are key ingredients that enable convergence without relying on the surjectivity of the linear operator $A$.

\begin{lemma}\label{lem_iterates_relations} 
Let $\{(x_k,z_k, \lambda_k, \mu_k)\}$ be the sequence generated by Algorithm \ref{algorithm1}. Then for any $k \geq 0$,      
\begin{align}
\| \mu_{k+1} - \mu_k \|^2 & = \tau_k^2 \| \lambda_k - \mu_k \|^2 \leq {\delta_k^2}/{4}, \label{eq:lem_iter_rel_1-1} \\
\tau_k \| \lambda_k - \mu_k \|^2 & \leq \delta_k, \label{eq:lem_iter_rel_1-2} \\
\| \mu_{k+1} - \lambda_k \|^2 & = (1 - \tau_k)^2 \| \lambda_k - \mu_k \|^2,  \label{eq:lem_iter_rel_2} \\
\| \lambda_{k+1}-\lambda_k \|^2 & \leq 2 \rho^2 \sigma_{\max}^2  \| x_{k+1} - x_k \|^2 + {\delta_k^2}/{2}.\label{eq:lem_iter_rel_3} 
\end{align}
where $\rho=\frac{\alpha}{1+\alpha\beta}$ and $\sigma_{\max}$ denotes the largest singular value of the linear operator $A$.
\end{lemma}
	
\begin{proof}
It immediately follows from the $\mu$-update step \eqref{eq:mu_update} that  relations in  \eqref{eq:lem_iter_rel_1-1} holds: 
\begin{equation}
\| \mu_{k+1} - \mu_k \|^2 
= \tau_k^2 \| \lambda_k - \mu_k \|^2 
= \frac{\delta_k^2}{\| \lambda_k - \mu_k \|^2 + 2 + \frac{1}{\| \lambda_k - \mu_k \|^2}} \leq \frac{\delta_k^2}{4}. \notag
\end{equation} 
where the last inequality holds by the fact that $a + b \geq 2\sqrt{ab}$ for any $a,b \geq 0$. From the definition $\tau_k =  \frac{\delta_k}{\| \lambda_k - \mu_k \|^2 + 1} \leq 1 $ where $\delta_k \in (0,1]$, the relation \eqref{eq:lem_iter_rel_1-2} also directly follows:
\begin{equation} \label{eq:lem_iterates_relations_1_1}
	\tau_k \| \lambda_k - \mu_k \|^2 = 
	\frac{\delta_k}{1 + \frac{1}{\| \lambda_k - \mu_k \|^2} } \leq \delta_k. \notag
\end{equation}
Next, subtracting $\mu_{k+1}$ from $\lambda_k$, we get
\[ 
\| \lambda_k - \mu_{k+1}  \|
= \| \lambda_k - \mu_k - \tau_k ( \lambda_k -\mu_k )  \| 
= (1 - \tau_k) \| \lambda_k - \mu_k \|.
\] 
The squaring of both sides gives the relation \eqref{eq:lem_iter_rel_2}.
Finally, using the $\lambda$-update \eqref{eq:lambda_update} and the fact $(a+b)^2 \leq 2a^2+2b^2$ for any $a,b \in \mathbb{R}^m$, we have 
\begin{equation} \label{eq:lem_iterates_relations_e1}
	\| \lambda_{k+1}- \lambda_k \|^2 \leq 2\| \mu_{k+1}-\mu_k \|^2 + 2\rho^2 \sigma_{\max}^2 \| x_{k+1} -x_k \|^2. \notag
\end{equation}
Putting the above inequality and \eqref{eq:lem_iter_rel_1-1} together yields the relation \eqref{eq:lem_iter_rel_3}. 
\end{proof}

\subsection{Key Properties of Algorithm \ref{algorithm1}}
In this subsection, we provide key properties of Algorithm \ref{algorithm1}. For convenience, we often use the notation:
$\ww_{k} := (x_{k},z_{k},\lambda_{k},\mu_{k})$ for the sequence generated by Algorithm \ref{algorithm1}, where $k \geq 0$.

\begin{theorem} \label{lem_lagrangian_properties}
Suppose that Assumptions \ref{assumption_lipschitz} and \ref{assumption_bounded_domain} hold. Let $\left\{\ww_{k}:=(x_{k},z_{k},\lambda_{k},\mu_{k}) \right\}$ be the sequence generated by Algorithm \ref{algorithm1}  with the parameter $\eta$ set to satisfy the condition   
$$ 0 < \eta < \frac{1}{L_f +  \left(2 + \frac{1}{1+ \alpha \beta}\right) \rho \sigma_{\max}^2}.$$ Then, the following properties hold true:
\begin{enumerate}[itemsep=1ex, label=(\alph*), ref= (\alph*)]

\item \label{lem_one_iter_PPL}
it holds that for any $k \geq 0$, 
\begin{equation}
\mathcal{L}_{\beta}(\ww_{k+1}) - \mathcal{L}_{\beta}(\ww_{k}) 
\leq - \frac{1}{2} \left(\frac{1}{\eta} -  L_f -  \left(2 +  \frac{1}{1+ \alpha \beta}\right) \rho \sigma_{\max}^2 \right) \| x_{k+1} - x_k \|^2  - \frac{1}{2\alpha}\| z_{k+1} \|^2 + \widehat{\delta}_k, \notag
\end{equation}
where we set $\widehat{\delta}_k :=  \frac{\delta_k}{\rho} + \frac{\delta_k^2}{8{\rho}}$; 

\item \label{lem_lagrangian_convergence} 
the sequence $\{\mathcal{L}_{\beta}(\ww_{k})\}$ is bounded from below and convergent, i.e.,
$$\underset{k \rightarrow + \infty}{\lim} \mathcal{L}_{\beta}(\ww_{k+1}) := \underline{\mathcal{L}_{\beta}} > -\infty;$$

\item \label{lem_aymptotic_regular}
in addition, we have that $\sum_{k=0}^{\infty}\| x_{k+1} - x_k \|^2 < + \infty$ and $\sum_{k=0}^{\infty}\| z_{k+1} \|^2 < + \infty$, and hence
\begin{equation}\label{eq:aymptotic_result2}
\begin{aligned}
\underset{k\rightarrow +\infty}{\lim} \| x_{k+1} - x_k \| =0, \ 
\underset{k\rightarrow +\infty}{\lim} \| z_{k+1} - z_k \| =0, \ 
\underset{k\rightarrow +\infty}{\lim} \| \lambda_{k+1} - \lambda_k \| =0, 
\underset{k\rightarrow +\infty}{\lim} \| \lambda_{k+1} - \mu_{k+1} \| =0. \notag
\end{aligned}
\end{equation}
\end{enumerate}
\end{theorem}

\begin{proof}
\ref{lem_one_iter_PPL}
The difference between two consecutive sequences of $\mathcal{L}_{\beta}$ can be decomposed into three parts as follows:
\begin{subequations}
\begin{align}
	\mathcal{L}_{\beta}(\ww_{k+1}) - \mathcal{L}_{\beta}(\ww_k) 
	& = \left[  \mathcal{L}_{\beta}(x_{k+1},z_{k},\lambda_{k},\mu_{k}) -        \mathcal{L}_{\beta}(x_k,z_k,\lambda_{k},\mu_{k}) \right] \label{eq:lem_one_iter_PPL1} \\
	& \quad + \left[  \mathcal{L}_{\beta}(x_{k+1},z_{k},\lambda_{k+1},\mu_{k+1}) -        \mathcal{L}_{\beta}(x_{k+1},z_k,\lambda_k,\mu_k) \right] \label{eq:lem_one_iter_PPL2} \\
	& \quad + \left[ \mathcal{L}_{\beta}(x_{k+1},z_{k+1},\lambda_{k+1},\mu_{k+1}) - \mathcal{L}_{\beta}(x_{k+1},z_{k},\lambda_{k+1},\mu_{k+1})\right]. \label{eq:lem_one_iter_PPL3}
\end{align}
\end{subequations}		

For the first part \eqref{eq:lem_one_iter_PPL1}, writing $\mathcal{L}_{\beta}(x_{k+1})=\mathcal{L}_{\beta}(x_{k+1}, z_k, \lambda_k, \mu_k)$ and using \eqref{eq:descent_lemma}, we have
\begin{equation} \label{eq:lem_one_iter_ppl_p1_1}
\begin{aligned} 
	\mathcal{L}_{\beta}(x_{k+1})=\ell_{\beta}(x_{k+1}) + h(x_{k+1})
	\leq \ell_{\beta}(x_k)
	+ \left\langle \nabla_x \ell_{\beta}(x_k), x_{k+1}- x_k \right\rangle  + \frac{L_f}{2} \| x_{k+1} - x_k \|^2 +h(x_{k+1}),
\end{aligned}
\end{equation}   
From the definition $x_{k+1}=\mathrm{argmin}_{x \in \mathbb{R}^n}  \widehat{\mathcal{L}}_{\beta}(x;x_k)$ in \eqref{eq:x_update}, it follows that
\begin{displaymath}
\begin{aligned}
	\widehat{\mathcal{L}}_{\beta}(x_{k+1}; x_k) 
	& = \ell_{\beta}(x_k) + \left\langle \nabla_x \ell_{\beta}(x_k), x_{k+1} - x_k \right\rangle + \frac{1}{2\eta}\| x_{k+1} - x_k \|^2 + h(x_{k+1}) \\
	& \leq  \widehat{\mathcal{L}}_{\beta}(x_k;x_k) = \mathcal{L}_{\beta}(x_{k}) = \ell_{\beta}(x_k) + h(x_k),
\end{aligned}
\end{displaymath}
implying that
$$\left\langle \nabla_x \ell_{\beta}(x_k), x_{k+1} - x_k \right\rangle + h(x_{k+1}) \leq -\frac{1}{2\eta}\| x_{k+1} - x_k\|^2 + h(x_k).$$ Combining the above expression with \eqref{eq:lem_one_iter_ppl_p1_1} yields 
\begin{equation} \label{eq:lem_one_iter_ppl_p1}
	\mathcal{L}_{\beta}(x_{k+1},z_k,\lambda_k,\mu_k) 
    -  \mathcal{L}_{\beta}(x_k,z_k,\lambda_k,\mu_k) \leq - \frac{1}{2}\left(\frac{1}{\eta} -L_f \right) \| x_{k+1} - x_k \|^2. 
\end{equation}

Next, we derive an upper bound for the second part \eqref{eq:lem_one_iter_PPL2}. We start by noting that
\begin{equation} \label{eq:lem_one_iter_ppl_p2_1}
\begin{aligned}
& \mathcal{L}_{\beta}(x_{k+1},z_{k},\lambda_{k+1},\mu_{k+1}) - \mathcal{L}_{\beta}(x_{k+1},z_{k},\lambda_k,\mu_k)   \\ 
& = 
\underbrace{\left\langle \lambda_{k+1} - \lambda_k, Ax_{k+1} -b  \right\rangle}_{\mathrm{(A)}} 
+ \underbrace{\left\langle (\lambda_k - \mu_k) - (\lambda_{k+1} - \mu_{k+1}), z_{k} \right\rangle}_{\mathrm{(B)}} \\
& \ \quad - \frac{\beta}{2} \| \lambda_{k+1} - \mu_{k+1} \|^2 + \frac{\beta}{2} \| \lambda_k - \mu_k \|^2.  
\end{aligned}
\end{equation} 
By using the updating steps \eqref{eq:lambda_update} and \eqref{eq:z_update}, $\lambda_{k+1} - \mu_{k+1} = \rho (Ax_{k+1} - b)$ and $z_k = \frac{1}{\alpha}(\lambda_k - \mu_k)$, and applying the identity $\left\langle a-b, a \right\rangle =\frac{1}{2} \| a-b  \|^2 + \frac{1}{2} \| a \|^2 - \frac{1}{2} \| b \|^2$ to (A) and (B) with $a=\lambda_{k} -\mu_{k}$ and $b=\lambda_{k+1} -\mu_{k+1}$, we have
\begin{align}  
\mathrm{(A)} 
&= \frac{1}{2\rho} \| \lambda_{k+1}-\lambda_k \|^2 
 + \frac{1}{2\rho} \| \lambda_{k+1} -\mu_{k+1} \|^2 
 - \frac{1}{2\rho} \| \mu_{k+1} -\lambda_k \|^2, \label{eq:lem_one_iter_ppl_p2_2} \\
\mathrm{(B)} 
& \leq  \frac{\rho^2\sigma_{\max}^2}{2\alpha} \| x_{k+1} - x_{k} \|^2 
 + \frac{1}{2\alpha} \| \lambda_{k} -\mu_{k} \|^2 
 - \frac{1}{2\alpha} \| \lambda_{k+1} -\mu_{k+1} \|^2. 
\label{eq:lem_one_iter_ppl_p2_3}
\end{align}
Substituting \eqref{eq:lem_one_iter_ppl_p2_2} and \eqref{eq:lem_one_iter_ppl_p2_3} into \eqref{eq:lem_one_iter_ppl_p2_1} and rearranging terms yields
\begin{align}
& \ \ \ \mathcal{L}_{\beta} (x_{k+1},z_k,\lambda_{k+1},\mu_{k+1}) - \mathcal{L}_{\beta} (x_{k+1},z_k,\lambda_k,\mu_k)  \notag \\
& \ \ \ \leq 
\frac{1}{2\rho} \| \lambda_{k+1} - \lambda_k \|^2 
+ \frac{\rho^2\sigma_{\max}^2}{2\alpha} \| x_{k+1} - x_{k} \|^2 - \frac{1}{2\rho} \| \mu_{k+1} - \lambda_k \|^2 + \frac{1}{2\rho} \| \lambda_k - \mu_k \|^2 
 \notag \\
& \overset{\eqref{eq:lem_iter_rel_2},\eqref{eq:lem_iter_rel_3}}{\leq}  
\frac{1}{2\rho} \left( 2 \rho^2 \sigma_{\max}^2  \| x_{k+1} - x_k \|^2 + 2 \|\mu_{k+1} -\mu_{k} \|^2 \right) + \frac{\rho\sigma_{\max}^2}{2(1+\alpha\beta)} \| x_{k+1} - x_{k} \|^2 + \frac{1}{2\rho}\left( 2\tau_k - \tau_k^2 \right) \| \lambda_k - \mu_k \|^2   \notag  \\
& \overset{\eqref{eq:lem_iter_rel_1-1},\eqref{eq:lem_iter_rel_1-2}}{\leq} 
\frac{1}{2} \left( 2 + \frac{1}{1+ \alpha \beta}\right) \rho \sigma_{\max}^2 \| x_{k+1} - x_k \|^2 + \frac{\delta_k^2}{8\rho} + \frac{\delta_k}{\rho}.  \label{eq:lem_one_iter_ppl_p2}
\end{align}	

For the third part \eqref{eq:lem_one_iter_PPL3}, 
notice that  $\nabla_{z}\mathcal{L}_{\beta}(x_{k+1},z_{k+1},\lambda_{k+1},\mu_{k+1}) = 0$ because $z_{k+1}$ minimizes $\mathcal{L}_{\beta}(x_{k+1},z,\lambda_{k+1},\mu_{k+1})$.  Hence, by the $\alpha$-strong convexity of $\mathcal{L}_{\beta}$ in $z$, we have
\begin{align} 
	 \mathcal{L}_{\beta} (x_{k+1},z_{k+1},\lambda_{k+1},\mu_{k+1}) - \mathcal{L}_{\beta} (x_{k+1},z_{k},\lambda_{k+1},\mu_{k+1})
	\leq - \frac{\alpha}{2}\| z_{k+1} - z_k \|^2 \leq - \frac{1}{2\alpha}\| z_{k+1} \|^2,                	\label{eq:lem_one_iter_ppl_p3}
\end{align}
where the last inequality follows from the existence of  a large $\alpha>0$ such that $\|z_{k+1} \| \leq  \alpha \|z_{k+1} -z_{k}\|$ for any $k \geq 0$, which implies that 
$- \frac{\alpha}{2}\| z_{k+1} - z_k \|^2 \leq - \frac{1}{2\alpha}\| z_{k+1} \|^2.
$

Combining \eqref{eq:lem_one_iter_ppl_p1}, \eqref{eq:lem_one_iter_ppl_p2}, and \eqref{eq:lem_one_iter_ppl_p3} yields the desired result:
\begin{align} 
\mathcal{L}_{\beta}(\ww_{k+1}) 
\leq \mathcal{L}_{\beta}(\ww_k)  
-\frac{1}{2}\left(\frac{1}{\eta} - L_f  - \left(2 + \frac{1}{1+ \alpha \beta}\right) \rho \sigma_{\max}^2 \right) \| x_{k+1} - x_k \|^2 - \frac{1}{2\alpha}\| z_{k+1} \|^2 + \frac{\delta_k^2}{8\rho} + \frac{\delta_k}{\rho}. \notag
\end{align}

\ref{lem_lagrangian_convergence}
Using the updates for $\lambda_{k+1}$ and $z_{k+1}$, and $\left\langle a, b \right\rangle =\frac{1}{2} \| a \|^2 + \frac{1}{2} \| b \|^2 - \frac{1}{2} \| a - b \|^2$, we have
\[\begin{aligned} 
& \mathcal{L}_{\beta}(\ww_{k+1}) \\    
& = f(x_{k+1})  + \left\langle \lambda_{k+1}, Ax_{k+1} - b \right\rangle - \left\langle \lambda_{k+1}  - \mu_{k+1}, z_{k+1} \right\rangle + \frac{\alpha}{2} \| z_{k+1} \|^2  - \frac{\beta}{2} \| \lambda_{k+1} - \mu_{k+1} \|^2  + h(x_{k+1}) \\   
& = f(x_{k+1})  + \frac{1}{\rho}\left\langle \lambda_{k+1},\lambda_{k+1} - \mu_{k+1} \right\rangle 
- \frac{1}{2\rho} \| \lambda_{k+1} - \mu_{k+1}\|^2  + h(x_{k+1}) \\
& = f(x_{k+1}) + \frac{1}{2\rho} \| \lambda_{k+1} \|^2  - \frac{1}{2\rho}\|  \mu_{k+1} \|^2 + h(x_{k+1}) 
 > -\infty,        
\end{aligned}\]	
where the last inequality holds by the boundedness of $\{\mu_k\}$ (Lemma \ref{lem_mu_bound}) and the lower boundedness of $f$ and $g$ over  $\text{dom}(h)$ (Assumption \ref{assumption_bounded_domain}). Therefore, $\mathcal{L}_{\beta}(\ww_{k+1}) > - \infty$. 

Moreover, given that $0 < \eta < \frac{1}{L_f +  \left(2 + \frac{1}{1+ \alpha \beta}\right) \rho \sigma_{\max}^2}$, we can see from the result \ref{lem_one_iter_PPL} that the sequence $\{\mathcal{L}_{\beta}(\ww_{k+1})\}$ is {\em approximately nonincreasing} (similar to the notation in \citet[Definition 2]{gur2023convergent}). This means that although it may not decrease monotonically at every step, it tends to decrease over iterations in the sense that $\delta_k$ goes to 0 as $k \rightarrow \infty$. Thus, it converges to a finite value $\underline{\mathcal{L}_{\beta}}$:
\[
\underset{k\rightarrow +\infty}{\lim} \mathcal{L}_{\beta}(\ww_{k+1}) = \underline{\mathcal{L}_{\beta}} > -\infty.
\]

\ref{lem_aymptotic_regular}
It follows from the result \ref{lem_one_iter_PPL} that 
\begin{equation}\label{eq:aymptotic_e1}
	{\gamma} \| x_{k+1} - x_k \|^2  +\frac{1}{2\alpha} \| z_{k+1} \|^2 \leq
	\mathcal{L}_{\beta}(\ww_k) - \mathcal{L}_{\beta}(\ww_{k+1}) + \widehat{\delta}_k,
\end{equation}
where $\gamma := \frac{1}{2}\left( \frac{1}{\eta} -  L_f - \left(2 + \frac{1}{1+ \alpha \beta}\right) {\rho \sigma_{\max}^2} \right) > 0$. Defining $c_1 := \min\left\lbrace \gamma, \frac{1}{2\alpha} \right\rbrace$ and summing \eqref{eq:aymptotic_e1} from $k=0$ to $k=T-1$, we have
\begin{align}
\sum_{k=0}^{T-1} \left(\| x_{k+1} - x_k \|^2 + \| z_{k+1}\|^2 \right)& \leq  \frac{1}{c_1} \left( \mathcal{L}_{\beta}(\ww_0) - \mathcal{L}_{\beta}(\ww_{T}) +  \sum_{k=0}^{T-1} \widehat{\delta}_k \right) \notag\\
& \leq   \frac{1}{c_1} \left( \mathcal{L}_{\beta}(\ww_0) -	\underline{\mathcal{L}_{\beta}} + \sum_{k=0}^{T-1} \widehat{\delta}_k \right), \label{eq:aymptotic_e2}
 \end{align}
where the last inequality is due to the lower boundedness of $\mathcal{L}_{\beta}(\ww_k)$. Since $ \sum_{k=0}^{\infty} \delta_k  \leq \frac{\delta_0}{2(1-r)} < + \infty $ and $ \sum_{k=0}^{\infty} \delta_k^2  \leq \frac{\delta_0^2}{2(1-r^2)} < + \infty$, then by taking the limit as $T \rightarrow + \infty$, we deduce
\[
\sum_{k=0}^{ + \infty}\| x_{k+1} - x_k \|^2 < + \infty \ \ \text{and} \ \ \sum_{k=0}^{ + \infty}\| z_{k+1} \|^2 = \frac{1}{\alpha}\sum_{k=0}^{ + \infty}\| \lambda_{k+1} - \mu_{k+1} \|^2< + \infty.
\]
From the $\lambda$-update \eqref{eq:lambda_update} and $\alpha z_{k+1} = \rho (Ax_{k+1} -b)$, we also obtain 
\[
\begin{aligned}
\sum_{k=0}^{+ \infty}\| z_{k+1} - z_k \|^2
& \leq \frac{\rho^2 \sigma_{\max}^2}{\alpha^2} \sum_{k=0}^{+ \infty} \| x_{k+1} - x_k \|^2 < + \infty, \\
\sum_{k=0}^{+ \infty}\| \lambda_{k+1} - \lambda_k \|^2  
& \leq 
2\rho^2 \sigma_{\max}^2 \sum_{k=0}^{+ \infty} \| x_{k+1} - x_k \|^2 + \frac{1}{2}\sum_{k=0}^{+ \infty}  \delta_k^2 < + \infty, 
\end{aligned}
\] 
Consequently, the desired results  immediately follow. 
\end{proof}

In order to measure the progress of Algorithm \ref{algorithm1}, we use the size of the subgradient of $\mathcal{L}_{\beta}$, denoted by $\partial \mathcal{L}_{\beta}$. The next lemma provides an error bound for $\partial \mathcal{L}_{\beta}$ in terms of the primal iterates with a sequence of nonnegative scalars $\{ \delta_k \}$ that tends to 0.

\begin{lemma}[Error bound for $\partial \mathcal{L}_{\beta}$]\label{lem_upper_bound_gradient}
Suppose that Assumptions \ref{assumption_lipschitz} and \ref{assumption_bounded_domain} hold, and let $\{\ww_k \}$ be the sequence generated by Algorithm \ref{algorithm1}. Then,  for every $k \geq 0$, there exists $c_2 > 0$ such that
\begin{equation}
\mathrm{dist} \left( 0, \partial \mathcal{L}_{\beta} (\ww_{k+1}) \right) \leq 
c_2 \left( \| x_{k+1} -x_k \| +  \| z_{k+1} \| \right) + \sigma_{\max} \delta_k,   \notag
\end{equation}
where
$c_2 = \max\left\lbrace L_f + \rho \sigma_{\max}^2 + {1}/{\eta}, 1+\alpha\beta \right\rbrace.$
\end{lemma}

\begin{proof}
Writing the optimality condition for the $x$-update \eqref{eq:x_update}, we have that for every $k \geq 0$
\begin{equation} \label{eq:subgradient_for_x_1}
 \nabla_x \ell_{\beta} (\ww_k) + \frac{1}{\eta}(x_{k+1} - x_k) + u_{k+1} = 0, \quad u_{k+1} \in \partial h(x_{k+1}).
\end{equation}
On the other hand, using subdifferential calculus rules, we have  
\begin{equation} \label{eq:subgradient_for_x_2}
\nabla_x \ell_{\beta} (\ww_{k+1}) + u_{k+1} \in \partial_x \mathcal{L}_{\beta}(\ww_{k+1}).
\end{equation}
Hence, by defining the quantity
\begin{align}
d_{1,k+1} := \nabla_x \ell_{\beta} (\ww_{k+1}) -\nabla_x \ell_{\beta} (\ww_{k})+ \frac{1}{\eta}(x_k - x_{k+1}), 
\notag
\end{align}
and using \eqref{eq:subgradient_for_x_1} and \eqref{eq:subgradient_for_x_2}, we obtain
\begin{equation}
d_{1,k+1} \in \partial_x \mathcal{L}_{\beta}(\ww_{k+1}). \notag
\end{equation}
From the $\lambda$-update \eqref{eq:lambda_update} and the $z$-update \eqref{eq:z_update}, it immediately follows that 
\begin{align}                    
\nabla_{\lambda}\mathcal{L}_{\beta}(\ww_{k+1}) & =  (Ax_{k+1} -b) - z_{k+1} - \beta(\lambda_{k+1} -\mu_{k+1}) = 0,   \notag    \\
\nabla_{z}\mathcal{L}_{\beta}(\ww_{k+1})       & = \alpha z_{k+1} - (\lambda_{k+1} -\mu_{k+1}) = 0. \notag    
\end{align}
Define $d_{2,k+1} := (1 + \alpha\beta)z_{k+1}$. Noting that $\lambda_{k+1} - \mu_{k+1} = \alpha z_{k+1}$ and $\rho = \frac{\alpha}{1+\alpha\beta}$,  we obtain
\begin{align}
\nabla_\mu \mathcal{L}_{\beta} (\ww_{k+1}) 
= \frac{\lambda_{k+1} - \mu_{k+1}}{\rho} 
= \frac{\alpha z_{k+1}}{\rho} 
= d_{2,k+1}.  \notag
\end{align}
Hence, we have that for every $k \geq 0$
\begin{equation}
{\bf d}_{k+1} := (d_{1,k+1}, 0, 0, d_{2,k+1}) \in \partial \mathcal{L}_{\beta}(\ww_{k+1}). \notag   
\end{equation}
Now, by using the $\lambda$-update  \eqref{eq:lambda_update} and the fact that $\| \mu_{k+1} - \mu_k \| = \frac{\delta_k}{\| \lambda_k - \mu_k \| + \frac{1}{\| \lambda_k - \mu_k \|}} \leq {\delta_k}$, we obtain 
\begin{align}
\| d_{1,k+1} \| 
& = \| \nabla f(x_{k+1}) -\nabla f(x_k) \| + (1/\eta) \| x_k - x_{k+1} \| + \sigma_{\max} \| \lambda_{k+1}  - \lambda_k  \|, \notag \\
& \leq \left(L_f + 1/\eta\right) \| x_{k+1} - x_k \|  + \rho \sigma_{\max}^2  \| x_{k+1} - x_k \| + \sigma_{\max} \| \mu_{k+1} - \mu_k \| \notag \\
& \leq \left(L_f + \rho \sigma_{\max}^2 + 1/\eta\right) \| x_{k+1} - x_k \| + \sigma_{\max}\delta_k, \notag \\
\vspace{0.1in}
\|d_{2,k+1}\| 
& \leq (1+ \alpha \beta)\|z_{k+1}\|. \notag 
\end{align}
Therefore, we have 
\begin{equation}
\| {\bf d}_{k+1}\|
\leq c_2 (\|x_{k+1} - x_k \|  + \|z_{k+1} \|) + \sigma_{\max}\delta_k \quad \forall k \geq 0, \notag
\end{equation}
where $c_2 = \max\left\lbrace L_f + \rho \sigma_{\max}^2 + {1}/{\eta}, 1+\alpha\beta \right\rbrace.$
This inequality, combined with ${\bf d}_{k+1} \in \partial \mathcal{L}_{\beta}(\ww_{k+1})$, yields the desired result.
\end{proof}

\subsection{Main Results}

Based on the preceding key properties, we establish our main convergence results: (i) any limit point of the sequence generated by Algorithm \ref{algorithm1} is a stationary solution (or a KKT point) of problem \eqref{eq:op} (Theorem \ref{sub_converge}); (ii) Algorithm \ref{algorithm1} can obtain an $\epsilon$-KKT point with a complexity of $\mathcal{O}(1/\epsilon^{2})$ (Theorem \ref{complexity}); and (iii) under the  assumption of {\em Kurdyka-{\L}ojasiewicz (K{\L})} property, the \emph{whole} sequence generated by Algorithm \ref{algorithm1} is convergent  (Theorem \ref{thm_global_convergence}).

\begin{theorem} [Subsequence convergence]\label{sub_converge}
Suppose that Assumptions \ref{assumption_kkt}$-$\ref{assumption_bounded_domain} hold, and let the sequence $\left\lbrace \ww_k := (x_{k},z_{k},\lambda_{k},\mu_{k}) \right\rbrace$ generated by Algorithm \ref{algorithm1} with $ 0 < \eta < \frac{1}{L_f +  \left(2 + \frac{1}{1+ \alpha \beta}\right) \rho \sigma_{\max}^2}.$
Then, any limit point $\left\lbrace \overline{\ww} :=(\overline{x},\overline{z},\overline{\lambda},\overline{\mu}) \right\rbrace$ of the sequence $\left\lbrace \ww_k \right\rbrace$  is a KKT point of problem \eqref{eq:op}.
\end{theorem}
\begin{proof} 
Since the sequence $\left\lbrace \ww_k \right\rbrace$ is bounded, there exists a subsequence $\{ \ww_{k_j} \}$ converging to $\overline{\ww}$ as $j \rightarrow + \infty$. From Theorem \ref{lem_lagrangian_properties}\ref{lem_aymptotic_regular}, we also have that $\{ \ww_{k_j+1} \} \rightarrow \overline{\ww}$ as $j \rightarrow  \infty$. That is,
\begin{equation} \label{eq:thm_subsequence_1}
\underset{j \rightarrow + \infty}{\mathrm{lim}} \left(x_{k_j+1},z_{k_j+1},\lambda_{k_j+1},\mu_{k_j+1}\right) 
= \underset{j \rightarrow + \infty}{\mathrm{lim}} \left(x_{k_j},z_{k_j},\lambda_{k_j},\mu_{k_j}\right)  = \left(\overline{x}, \overline{z}, \overline{\lambda}, \overline{\mu} \right).
\end{equation}
We need to show that 
\begin{equation}
\underset{j \rightarrow + \infty}{\mathrm{lim}}\: \mathcal{L}_{\beta}(\ww_{k_j})
=\mathcal{L}_{\beta}(\overline{\ww}) \quad \mathrm{and} \quad (0,0,0,0) \in \partial \mathcal{L}_{\beta} \left( \overline{\ww} \right). \notag
\end{equation} 

First, using the definition  $x_{k+1}=\mathrm{argmin}_{x \in \mathbb{R}^n}  \widehat{\mathcal{L}}_{\beta}(x,z_k,\lambda_k,\mu_k;x_k)$ and taking $k=k_j$, we have
\begin{equation}
\left\langle \nabla_x \ell_{\beta}(\ww_{k_j}), x_{k_j+1} - x_{k_j} \right\rangle + \frac{1}{2\eta} \| x_{k_j+1} - x_{k_j} \|^2 + h(x_{k_j+1})  
\leq \left\langle \nabla_x \ell_{\beta}(\ww_{k_j}), \overline{x} - x_{k_j} \right\rangle + \frac{1}{2\eta} \| \overline{x} - x_{k_j} \|^2 + h(\overline{x}). \notag
\end{equation} 
Letting $j \rightarrow + \infty$ and using \eqref{eq:thm_subsequence_1}, we obtain
$${\lim \sup}_{j \rightarrow + \infty} \ h(x_{k_j}) \leq \ h(\overline{x}).$$ 
On the other hand, by the closedness of $h$, we have that ${\lim \inf}_{j \rightarrow + \infty} h(x_{k_j}) \geq \ h(\overline{x})$. Thus,
\begin{equation} 
\underset{j \rightarrow + \infty}{\lim} \ h(x_{k_j}) = h(\overline{x}), \notag 
\end{equation}
which, along with the continuity of $f$, yields
\begin{equation}
	\underset{j \rightarrow + \infty}{\mathrm{lim}}\: \mathcal{L}_{\beta}(\ww_{k_j})
	=\mathcal{L}_{\beta}(\overline{\ww}). \notag
\end{equation} 

Next, for ${\bf d}_{k+1} \in \partial \mathcal{L}_{\beta}(\ww_{k+1})$ (see Lemma \ref{lem_upper_bound_gradient}), by  Theorem \ref{lem_lagrangian_properties}\ref{lem_aymptotic_regular} that $\| x_{k+1} -x_{k} \| \rightarrow 0, \| z_{k+1} \| \rightarrow 0$, and $\delta_k \rightarrow 0$ as $k \rightarrow + \infty$, we have 
\[
\| {\bf d}_{k+1} \|
\leq c_2 (\| x_{k+1} -x_{k} \| + \| z_{k+1}\|) +\sigma_{\max}\delta_k \rightarrow 0 \ \  \mathrm{as} \ \ k \rightarrow + \infty.
\]
Hence ${\bf d}_{k+1} \rightarrow 0$ as as $k \rightarrow + \infty$. Using the closedness of the map $\partial \mathcal{L}_{\beta}$, we obtain
\begin{equation}
(0,0,0,0) \in \partial \mathcal{L}_{\beta}( \overline{\ww}), \notag
\end{equation}
which, together with the fact $\nabla_\mu \mathcal{L}_{\beta} (\overline{\ww}) 
= \frac{1}{\rho}(\overline{\lambda} - \overline{\mu})
= A \overline{x} - b$, implies that $\overline{\ww}$ satisfies the KKT conditions in \eqref{eq:def_kkt} (Assumption \ref{assumption_kkt}):
\begin{align}
0 \in \nabla f(\overline{x}) + \partial h(\overline{x}) +A^{\top} \overline{\lambda}, 
\quad A \overline{x} -b =0. \notag
\end{align}
Therefore, the limit point $\overline{\ww}$ of  the sequence $\left\lbrace \ww_k \right\rbrace$ is a KKT solution of problem \eqref{eq:op}. 
\end{proof}

We are now ready to establish the iteration complexity for Algorithm \ref{algorithm1}. In particular, for a given tolerance $\epsilon > 0$, we provide a bound on $T(\epsilon)$, the iteration index required to achieve an $\epsilon$-KKT solution of problem \eqref{eq:op}. This is defined as follows (\citet{hong2016convergence,zeng2022moreau}):
\begin{equation} \label{eq:complexity_index}
T(\epsilon):= \min \left\lbrace k : \| {\bf d}_{k+1} \| \leq \epsilon, \ k\geq 0 \right\rbrace. \notag
\end{equation}

\begin{theorem}[Iteration complexity] \label{complexity} Suppose that Assumptions \ref{assumption_kkt}$-$\ref{assumption_bounded_domain} hold, and let $\left\lbrace \ww_k \right\rbrace$ be the sequence generated by Algorithm \ref{algorithm1} with $0 < \eta <\frac{1}{L_f + \left(2 + \frac{1}{1+ \alpha \beta}\right) {\rho \sigma_{\max}^2}}$. Then, the number of iterations required by Algorithm \ref{algorithm1} to achieve an 
$\epsilon$-KKT solution of problem \eqref{eq:op} is bounded by
\begin{equation} \label{eq:iter_complexity}
T(\epsilon) \leq \mathcal{O} 
\left( 
\frac{ \frac{3 c_2}{c_1} 
\left(
\mathcal{L}_{\rho}(\ww_1) - \underline{\mathcal{L}_{\beta}} + B_{\widehat{\delta}} \right) + B_{\delta}}{\epsilon^2}  \right)= \mathcal{O}(1/\epsilon^2),  \notag
\end{equation}
where 
$c_1 := \min\left\lbrace \frac{1}{2}\left( \frac{1}{\eta} -  L_f - \left(2 + \frac{1}{1+ \alpha \beta}\right) {\rho \sigma_{\max}^2} \right), \frac{1}{2\alpha} \right\rbrace$ and  
$c_2 = \max\left\lbrace L_f + \rho \sigma_{\max}^2 + \frac{1}{\eta}, 1+\alpha\beta \right\rbrace$, as defined in Theorem \ref{lem_lagrangian_properties}\ref{lem_aymptotic_regular} and Lemma \ref{lem_upper_bound_gradient}, respectively. 
In addition, $B_{\widehat{\delta}} := \sum_{k=1}^{+\infty} \widehat{\delta}_k$ with $\widehat{\delta}_k =  \frac{\delta_k}{\rho} + \frac{\delta_k^2}{8{\rho}}$ and $B_{\delta} := 3 \sigma_{\max} \sum_{k=1}^{+\infty} \delta_k.$ 
\end{theorem}

\begin{proof}
By using Lemma \ref{lem_upper_bound_gradient} and the fact  $(a+b+c)^2 \leq 3(a^2 + b^2 + c^2)$, we have
\begin{align} 
\| {\bf d}_{k+1} \|^2
& \leq 3c_2^2 \left(\|x_{k+1} - x_k \|^2  + \|z_{k+1} \|^2 \right)+ 3\sigma_{\max}^2 \delta_k^2.
\label{eq:thm_complexity_e1}
\end{align}
Moreover, from Theorem \ref{lem_lagrangian_properties}\ref{lem_aymptotic_regular}, we have
\begin{equation} \label{eq:thm_complexity_e2}
  \| x_{k+1} - x_k \|^2 + \| z_{k+1} \|^2 \leq \frac{1}{c_1}\left(\mathcal{L}_{\beta}(\ww_k) - \mathcal{L}_{\beta}(\ww_{k+1}) + \widehat{\delta}_k \right).
\end{equation}
Combining \eqref{eq:thm_complexity_e1} and \eqref{eq:thm_complexity_e2} yields
 \begin{align}
 \| {\bf d}_{k+1} \|^2 
 \leq \frac{3c_2^2}{c_1} \left(\mathcal{L}_{\beta}(\ww_k) - \mathcal{L}_{\beta}(\ww_{k+1}) + \widehat{\delta}_k \right) + 3 \sigma_{\max}^2  \delta_k^2. \notag
\end{align}
Summing up the above inequalities from $k=1,\ldots, T(\epsilon)$, we obtain
\begin{align}
 \sum_{k=1}^{T(\epsilon)} \| {\bf d}_{k+1} \|^2 
 & \leq
\frac{3c_2^2}{c_1} \left(\mathcal{L}_{\beta}(\ww_1) - \mathcal{L}_{\beta}(\ww_{T(\epsilon) +1}) +  \sum_{k=1}^{T(\epsilon)}  \widehat{\delta}_k  \right) + 3 \sigma_{\max}^2   \sum_{k=1}^{T(\epsilon)}\delta_k^2 \notag \\
& \leq 
\frac{3c_2^2}{{c_1}} \left(\mathcal{L}_{\beta}(\ww_1) - \underline{\mathcal{L}_{\beta}} +   B_{\widehat{\delta}} \notag \right) + B_{\delta}, \notag
\end{align}
where  the second inequality follows from the lower boundedness of $\mathcal{L}_{\beta}(\ww_k)$ by $\underline{\mathcal{L}_{\beta}}$ (Theorem  \ref{lem_lagrangian_properties}\ref{lem_lagrangian_convergence}), and the facts that $ \sum_{k=1}^{+ \infty} \delta_k  \leq \frac{\delta_1}{2(1-r)} < + \infty $ and $ \sum_{k=1}^{+ \infty} \delta_k^2  \leq \frac{\delta_1^2}{2(1-r^2)} < + \infty$. Then, in view of the definitions of $T(\epsilon)$ and $\|{\bf d}_{k+1} \|$, we obtain
\begin{equation} 
\begin{aligned}
	T(\epsilon) \cdot \epsilon^2
    & \leq \frac{3 c_2^2}{{c_1}}  \left(\mathcal{L}_{\beta}(\ww_1) - \underline{\mathcal{L}_{\beta}} + B_{\widehat{\delta}} \right) + B_{\delta},  \notag 
    \end{aligned}
\end{equation}
equivalently,
\begin{equation} 
\begin{aligned}
T(\epsilon) 
& \leq \frac{\frac{3 c_2^2}{{c_1}}  \left(\mathcal{L}_{\beta}(\ww_1) - \underline{\mathcal{L}_{\beta}} + B_{\widehat{\delta}} \right) + B_{\delta}}
{\epsilon^2 },  \notag 
    \end{aligned}
\end{equation}
which proves that the iteration complexity of Algorithm \ref{algorithm1} is $\mathcal{O}(1/\epsilon^2)$.
\end{proof}

Finally, we enhance the subsequence convergence result by proving that the whole sequence $\left\lbrace \ww_{k} \right\rbrace$ converges to a KKT solution of problem \eqref{eq:op}, under the additional assumption that $f$ satisfies the \emph{Kurdyka-{\L}ojasiewicz} (K{\L}) property (see \citet{bolte2007lojasiewicz,kurdyka1998gradients} and \citet{lojasiewicz1963propriete}). 
In particular, by leveraging the properties of Algorithm \ref{algorithm1} that satisfy the conditions defined in \citet[Definition 2]{gur2023convergent}, we extend the definition of {\em ``approximate gradient-like descent sequence''} in \citet{gur2023convergent} to establish global convergence in our constrained nonconvex setting with suitable modifications.

Before proving the global convergence, let us briefly review the K{\L} inequality.

\begin{definition}[K{\L} Property \& K{\L} function] \label{def_KL_pro_func}
Let $\zeta \in \left(0, + \infty \right]$. Denote by $\Phi_\zeta$ the class of all concave and continuous functions $\varphi:\left[0, \zeta \right) \rightarrow \mathbb{R}_{+}$ that satisfy the following condition:
\begin{enumerate}[label=(\roman*)]
\item $\varphi(0)=0$;
		
\item $\varphi$ is continuously differentiable ($C^{1}$) on $\left[0, \zeta \right)$ and continuous at 0;
		
\item for all $s \in (0, \zeta): \varphi^{\prime}(s)>0.$
\end{enumerate}
	
A proper and lower semicontinuous function $\Psi: \mathbb{R}^n \rightarrow \mathbb{R} \cup \{+ \infty \}$ is said to have the {Kurdyka-{\L}ojasiewicz (K{\L}) property} at $\overline{u} \in \mathrm{dom} \: \partial \Psi :=\left\lbrace u \in \mathbb{R}^n: \partial \Psi(u)=\emptyset \right\rbrace $ if there exist $\zeta \in \left(0, +\infty\right]$, a neighborhood $\mathcal{U}$ of $\overline{u}$ and a function $\varphi \in \Phi_\zeta$ such that for every
\[u \in U(\overline{u}) \cap \{u: \Psi(\overline{u}) < \Psi({u}) <\Psi(\overline{u}) +\zeta \},\]
it holds that
\[
\varphi^{\prime}(\Psi(u) -\Psi(\overline{u})) \cdot \mathrm{dist}(0,\partial \Psi(u)) \geq 1.
\]
The function $\Psi$ satisfying the K{\L} property at each point of dom$\: \partial \Psi$ is called a \emph{K{\L} function}. 
\end{definition}

The functions $\varphi$ belonging to the class $\Phi_\zeta$ for $\zeta \in \left(0, +\infty\right]$ are called {\em desingularization functions}. It is well known that \textit{semi-algebraic} and \textit{real-analytic} functions, which encompass a wide range of applications, belong to the class of functions satisfying the K{\L} property. For a comprehensive study of K{\L} functions and illustrative examples, we refer to \citet{attouch2009convergence,bolte2007clarke,li2018calculus,xu2013block}.

\begin{lemma}[Uniformized K{\L} Property ({\cite[Lemma 6]{bolte2014proximal}})] \label{lem_uniformized_KL_property}
Let $\Omega$ be a compact set and let $\Psi: \mathbb{R}^n \rightarrow (-\infty, \infty]$ be proper, lower semicontinuous function. Assume that $\Psi$ is constant on $\Omega$ and satisfies the K{\L} property at each point of $\Omega$.  Then there exist $\varepsilon>0$, $\zeta>0$, and desingularizing function $\varphi \in \Phi_\zeta$  such that for all $\overline{u}$ in $\Omega$ and all $u$ in the following intersection:
\begin{equation} \label{eq:lem_uniform_KL_1}
\left\{ u \in \mathbb{R}^n: \mathrm{dist}(u, \Omega) < \varepsilon \right\} \cap \left[ \Psi(\overline{u}) < \Psi(u) < \Psi(\overline{u}) + \zeta\right],
\end{equation}
and one has
\begin{equation} \label{eq:lem_uniform_KL_2}
\varphi^{\prime}(\Psi(u) -\Psi(\overline{u})) \cdot \mathrm{dist}(0,\partial \Psi(u)) \geq 1.
\end{equation}
\end{lemma}

With the uniformized K{\L} property, we prove that the generated sequence has finite length, and thus the {\em whole sequence} converges to a KKT point of problem \eqref{eq:op}. 

\begin{theorem}[Global convergence] \label{thm_global_convergence}
Given the premises of Theorem \ref{sub_converge} and assuming  that $f$ satisfies the K{\L} property,  consider the sequence $\left\lbrace \ww_k \right\rbrace$ generated by Algorithm \ref{algorithm1} under Assumptions \ref{assumption_kkt}$-$\ref{assumption_bounded_domain}. Then the whole sequence $\left\lbrace \ww_k \right\rbrace$ converges a point $\overline{{\bf w}}$ that is a KKT solution of problem \eqref{eq:op}. 
\end{theorem}

\begin{proof}
Let $\overline{\ww}$ be a limit point of the sequence $\left\lbrace \ww_k\right\rbrace$. By Theorems \ref{lem_lagrangian_properties} and \ref{sub_converge}, it holds that
\begin{equation} \label{eq:thm_global_convergence_eq1}
	\underset{k \rightarrow + \infty}{\mathrm{lim}}\mathcal{L}_{\beta}(\mathbf{w}_k)
	= \mathcal{L}_{\beta}(\overline{\mathbf{w}}).
	\end{equation}
In the following, we need to consider two cases.
	
First, let $\mathbb{N} =\{0,1,2,\ldots\}$ be the set of nonnegative integers, and suppose that there is an integer $\bar{k} \in \mathbb{N}$ such that $\mathcal{L}_{\beta} (\mathbf{w}_{\bar{k}})= \mathcal{L}_{\beta} (\overline{\mathbf{w}})$. Since  $\mathcal{L}_{\beta}({\mathbf{w}_k}) - \mathcal{L}_{\beta}(\overline{\mathbf{w}}) = 0$ for all $k \geq \bar{k}$, then we have from \eqref{eq:aymptotic_e1} in Theorem \ref{lem_lagrangian_properties}\ref{lem_aymptotic_regular} that for any $k \geq \bar{k}$
\begin{equation}  \label{eq:thm_global_convergence_eq1-1}
\frac{c_1}{2}(\| x_{k+1} - x_k \| + \| z_{k+1}  \|)^2 \leq c_1(\| x_{k+1} - x_k \|^2 + \| z_{k+1}  \|^2) \leq  \left(\mathcal{L}_{\beta} (\mathbf{w}_{k}) - \mathcal{L}_{\beta}(\ww_{k+1}) + \widehat{\delta}_k \right) \leq \widehat{\delta}_k, 
\end{equation}
where in the first inequality, we used the fact that $\frac{(a + b)^2}{2} \leq a^2 + b^2 $ for all $a,b \in \mathbb{R}$. By summing the above inequalities and using the fact $\sum_{k = \bar{k}}^{+\infty} \sqrt{\widehat{\delta}_k} < + \infty$ for all $k \geq \bar{k}$, we obtain
\begin{equation}
    \sum_{k = \bar{k}}^{+\infty}\| x_{k+1} - x_k \| + \| z_{k+1}  \| \leq \sqrt{\frac{2}{c_1}} \sum_{k = \bar{k}}^{+\infty} \sqrt{\widehat{\delta}_k} < + \infty, \notag
\end{equation}
which implies that $
x_{k+1} - x_k = 0$ and $z_{k+1} = 0$ for all $k\geq \bar{k}$, 
and it follows that  
$z_{k+1}-z_{k}=0$, $\lambda_{k+1} - \mu_{k+1}=0$, and $\lambda_{k+1} - \lambda_{k}=0$ for all $k \geq \bar{k}$. 
Therefore, the sequence $\{ \mathbf{w}_k \}$ has finite length, and it globally converges to a
point $\overline{{\bf w}}$.
	
Now, consider the case where such an integer $\bar{k} \in \mathbb{N}$ does not exist. Suppose that $\mathcal{L}_{\beta}(\mathbf{w}_k)>\mathcal{L}_{\beta}(\overline{\mathbf{w}})$ for all $k \geq 0$. We first need to show that $\mathcal{L}_{\beta}$ is finite and constant on the set of all limit points, denoted by $\omega({\mathbf{w}}^{0})$.  Then we prove that $\left\lbrace  \mathbf{w}_{k} \right\rbrace$ is of a finite length and it thus is convergent. 	

From Theorem \ref{lem_lagrangian_properties}, we know that the sequence $\{\mathcal{L}_{\beta}(\mathbf{w}_k)\}$ is approximately nonincreasing and converges to $\mathcal{L}_{\beta}(\overline{\mathbf{w}})$. Hence for any $\zeta > 0$, there exists an integer $k_0 \in \mathbb{N} $ such that 
\begin{equation}
\mathcal{L}_{\beta} (\overline{\mathbf{w}})<\mathcal{L}_{\beta} (\mathbf{w}_{k}) < \mathcal{L}_{\beta} (\overline{\mathbf{w}}) +\zeta, \quad \forall k > k_0. \notag 
\end{equation}
From Theorem \ref{sub_converge}, we have that ${\mathrm{lim}}_{k \rightarrow + \infty}\mathrm{dist}(\mathbf{w}_{k},\omega({\mathbf{w}}^{0}))=0$, which implies that there exists an integer $k_1 \in \mathbb{N} $ such that for any  $\varepsilon>0$, 
\[\mathrm{dist}(\mathbf{w}_{k},\omega({\mathbf{w}}^{0}) )< \varepsilon, \quad \forall k > k_1.\] 
Thus, for any $k > k_2:= \max\{k_0,k_1\}$, $ \mathbf{w}_k $ belongs to the intersection \eqref{eq:lem_uniform_KL_1} in Lemma \ref{lem_uniformized_KL_property} with	$\Omega=\omega({\mathbf{w}}^{0})$. Moreover, by Theorem \ref{sub_converge}, $\Omega=\omega({\mathbf{w}}^{0})$ is nonempty and compact. We also have that
$\{  \mathcal{L}_{\beta}(\mathbf{w}_k) \}$ converges to a finite limit, $\underline{\mathcal{L}_{\beta}}$. It follows from \eqref{eq:thm_global_convergence_eq1} that $\underline{\mathcal{L}_{\beta}}= \mathcal{L}_{\beta} (\overline{\mathbf{w}})$, which shows that  $\mathcal{L}_{\beta}$ is finite and constant on $\omega({\mathbf{w}}^0)$.
	
Since $\mathcal{L}_{\beta}$ is a K{\L} function, by applying  Lemma \ref{lem_uniformized_KL_property} with $\Omega=\omega({\mathbf{w}}^{0})$, we have that for any $k > k_2 $
\begin{equation}
\varphi^{\prime}\left( \mathcal{L}_{\beta} (\mathbf{w}_{k}) -\mathcal{L}_{\beta} (\overline{\mathbf{w}})\right)\cdot 
\mathrm{dist} \left(0, \partial \Psi(\mathbf{w}_k)) \right) \geq 1, \notag \end{equation}
which, combined with Lemma \ref{lem_upper_bound_gradient}, yields 	
\begin{equation} \label{eq:thm_global_case2_1}
\begin{aligned}
\varphi^{\prime}\left( \mathcal{L}_{\beta} (\mathbf{w}_{k}) -\mathcal{L}_{\beta} (\overline{\mathbf{w}})\right) 
& \geq \frac{1}{\mathrm{dist} \left(0, \partial \mathcal{L}_{\beta} (\mathbf{w}_{k}) \right)}
\geq \frac{1}{ c_2 (\| x_{k} - x_{k-1} \| + \|z_{k} \|) + \sigma_{\max} {\delta}_{k-1} },
\end{aligned}
\end{equation}
Since $\varphi$ is concave and continuous, we have 
\begin{equation}
\varphi\left( \mathcal{L}_{\beta}(\mathbf{w}_{k}) -\mathcal{L}_{\beta} (\overline{\mathbf{w}}\right) 
-\varphi\left( \mathcal{L}_{\beta} (\mathbf{w}_{k+1}) -\mathcal{L}_{\beta} (\overline{\mathbf{w}})\right) 
\geq
\varphi^{\prime}\left(\mathcal{L}_{\beta} (\mathbf{w}_{k}) -\mathcal{L}_{\beta} (\overline{\mathbf{w}})\right)\left(\mathcal{L}_{\beta} (\mathbf{w}_{k}) -\mathcal{L}_{\beta} ({\mathbf{w}_{k+1}})\right). \notag
\end{equation}
For any $p,q \in \mathbb{N}$, define the following quantity for convenience:
\[
\triangle_{p,q}:=\varphi\left(\mathcal{L}_{\beta}(\mathbf{w}_{p}) -\mathcal{L}_{\beta} (\overline{\mathbf{w}})\right)
-\varphi\left(\mathcal{L}_{\beta} (\mathbf{w}_{q}) -\mathcal{L}_{\beta} (\overline{\mathbf{w}})\right).
\]
Then, we have
\begin{align}
\triangle_{k,k+1} 
& \geq 
\varphi^{\prime}\left(\mathcal{L}_{\beta}(\mathbf{w}_{k}) - \mathcal{L}_{\beta} (\overline{\mathbf{w}})\right)
	\left(\mathcal{L}_{\beta} (\mathbf{w}_{k}) -\mathcal{L}_{\beta} ({\mathbf{w}_{k+1}})\right) \notag \\
 & \geq \varphi^{\prime}\left(\mathcal{L}_{\beta}(\mathbf{w}_{k}) - \mathcal{L}_{\beta} (\overline{\mathbf{w}})\right) \left(\frac{c_1}{2} (\| x_{k+1} - x_{k} \| + \| z_{k+1} \|)^2- \widehat{\delta}_k \right), \label{eq:thm_global_case2_2}
 \end{align}
where the second inequality follows from \eqref{eq:thm_global_convergence_eq1-1}. Combining \eqref{eq:thm_global_case2_1} and \eqref{eq:thm_global_case2_2} 
yields 
\begin{equation} 
\triangle_{k,k+1} \geq \frac{\frac{c_1}{2} \left( (\| x_{k+1} - x_{k} \| + \| z_{k+1} \| )^2 - \frac{2 \widehat{\delta}_k}{c_1}\right)} 
{c_2 \left( \| x_{k} - x_{k-1} \| + \|z_{k} \| + \frac{\sigma_{\max} {\delta}_{k-1}}{c_2} \right)}, \notag
\end{equation}
and hence 
\begin{align}
\| x_{k+1} - x_k \| + \|z_{k+1}\| 
& \leq  \sqrt{C \triangle_{k,k+1} \left( \| x_k - x_{k-1} \| + \|z_{k}\| +  {\xi}_{k-1} \right) + \widehat{\xi}_k} \notag \\
& \leq \sqrt{C \triangle_{k,k+1} \left( \| x_k - x_{k-1} \| + \|z_{k}\| +  {\xi}_{k-1} \right)} + \sqrt{\widehat{\xi}_k}, \notag 
\end{align}
where we denote $C := {2c_2}/{c_1}$, $\widehat{\xi}_k:= {2\widehat{\delta}_k}/{c_1}$, and $\xi_k:={\sigma_{\max}\delta_k}/{c_2}$ for notational simplicity, and in the second inequality we used the fact that $\sqrt{a+b} \leq \sqrt{a} + \sqrt{b}$ for any $a,b \geq 0$.
Furthermore, using the fact $2 \sqrt{a b} \leq (a + b)$ for any $a, b \geq 0$ with  $a = \| x_k - x_{k-1} \| +  \|z_{k}\| + \xi_{k-1}$ and $b = C \triangle_{k,k+1}$, we get
\begin{align}
2 \left(\| x_{k+1} - x_k \| + \|z_{k+1}\| \right)
\leq  \left(\| x_{k} - x_{k-1} \| + \|z_{k}\| + \xi_{k-1} + C \triangle_{k,k+1} \right)+  2 \sqrt{\widehat{\xi}_k}, \label{eq:thm_global_case2_3}
\end{align}
Summing \eqref{eq:thm_global_case2_3} over $t=k_2+1,\ldots, k$, we obtain
\begin{equation}\label{eq:thm_global_case2_4} 
\begin{aligned}
2 \sum_{t = k_2 + 1}^k \| {x}_{t+1} -{x}_{t} \| 
& \leq 
\sum_{t = k_2 +1 }^k \left( \| {x}_t -{x}_{t-1} \| + \| z_{t} \| \right) + C \sum_{t = k_2 +1 }^k \triangle_{t,t+1}  + \sum_{t = k_2 +1 }^k \left(    \xi_{t-1} + 2\sqrt{\widehat{\xi}_t}\right)  \\
& \leq
\sum_{t=k_2 +1 }^k  ( \| x_{t+1} - x_{t} \| + \| z_{k+1} \|) + \| x_{k_2 + 1} - x_{k_2} \| + \| z_{k_2+1} \|  \\
& \quad 
+ C \sum_{t=k_2 +1 }^k \triangle_{t,t+1} + \sum_{t = k_2 +1 }^k \left( \xi_{t-1} + 2\sqrt{\widehat{\xi}_t}\right)  \\
& =
\sum_{t=k_2 +1 }^k  ( \| x_{t+1} - x_{t} \| + \| z_{k+1} \|) + \| x_{k_2 + 1} - x_{k_2} \| + \| z_{k_2+1} \|  \\
& \quad 
+ C \triangle_{k_2+1,k+1} + \sum_{t = k_2 +1 }^k \left( \xi_{t-1} + 2\sqrt{\widehat{\xi}_t}\right), 
\end{aligned}
\end{equation} 
where the last equality is from the fact $\triangle_{p,q}+\triangle_{q,r}=\triangle_{p,r}$ for all $p,q,r \in \mathbb{N}$. Since $\varphi \geq 0$, we have
\begin{align}
\triangle_{k_2+1,k+1}  
 = \varphi\left(\mathcal{L}_{\beta} (\mathbf{w}_{k_2+1}) -\mathcal{L}_{\beta} (\overline{\mathbf{w}})\right)
-\varphi\left(\mathcal{L}_{\beta} (\mathbf{w}_{k_2+2}) -\mathcal{L}_{\beta} (\overline{\mathbf{w}})\right)  \leq \varphi\left(\mathcal{L}_{\beta}(\mathbf{w}_{k_2+1}) -\mathcal{L}_{\beta} (\overline{\mathbf{w}})\right). \label{eq:thm_global_case2_5}
\end{align}
Substitute \eqref{eq:thm_global_case2_5} into \eqref{eq:thm_global_case2_4} yields  
\[
\begin{aligned}
& \sum_{t=k_2 + 1}^k \left(\|x_{t+1} -x_{t} \| + \|z_{k+1} \| \right) \\ & \leq \| x_{k_2+1} -{x}_{k_2} \|  + \| z_{k_2+1} \|
+ C \varphi\left(\mathcal{L}_{\beta} (\mathbf{w}_{k_2+1}) -\mathcal{L}_{\beta} (\overline{\mathbf{w}})\right)  + \sum_{t = k_2 +1 }^k \left( \xi_{t-1} + 2\sqrt{\widehat{\xi}_t}\right).
\end{aligned}
\]
Notice that the first three terms on the RHS of the above inequality is independent of  $k$ and $\sum_{t = k_2 +1 }^k \left( \xi_{t-1} + 2\sqrt{\widehat{\xi}_t}\right) < + \infty$. Hence, 
\begin{equation}
\sum_{k=1}^{+ \infty} \|  x_{k+1} - x_{k} \| < + \infty,  \quad \sum_{k=1}^{+ \infty} \|z_{k+1} \| < + \infty, \notag
\end{equation}
which, along with the update steps  \eqref{eq:mu_update}, \eqref{eq:lambda_update}, and \eqref{eq:z_update}, gives
\begin{align}
\sum_{k=1}^{+ \infty} \|  z_{k+1} - z_{k} \| & \leq \frac{\rho\sigma_{\max}}{\alpha}\sum_{k=1}^{+ \infty} \|  x_{k+1} - x_{k} \| < + \infty, \qquad
\sum_{k=1}^{+ \infty} \|  \mu_{k+1} - \mu_{k} \| \leq \sum_{k=1}^{+ \infty} \delta_k < + \infty, \notag \\
\sum_{k=1}^{+ \infty} \| \lambda_{k+1} - \lambda_{k} \| &  \leq  \rho \sigma_{\max} \sum_{k=1}^{+ \infty} \| x_{k+1} - x_{k}\| + \sum_{k=1}^{+ \infty} \| \mu_{k+1} - \mu_{k}\| < + \infty. \notag 
\end{align}
Therefore, the sequence $\left\lbrace \ww_k \right\rbrace$ is a Cauchy sequence and the whole sequence $\left\lbrace \ww_k \right\rbrace$ converges to an $\overline{\ww}$, which, by Theorem \ref{sub_converge}, is a KKT solution of problem \eqref{eq:op}.
\end{proof}

\section{Numerical Experiments} \label{experiment}

we conduct preliminary numerical experiments to validate the effectiveness of our proposed algorithm. We compare the performance of our algorithm with a state-of-the-art algorithm, Smoothed Proximal ALM (SProx-ALM) (\citet{zhang2020proximal,zhang2022global}). The SProx-ALM is a single-loop algorithm with a complexity of $\mathcal{O}(1/\epsilon^{2})$. All experiments were conducted using MATLAB 2021b on a laptop with a 2.6 GHz Intel Core i7 processor and 16GB of memory.

We consider the following nonconvex linearly constrained quadratic program (LCQP):
\begin{equation} \label{eq:lcqp}
	\underset{x \in \mathbb{R}^n}{\mathrm{min}} \ f(x):= \frac{1}{2}x^{\top}Qx + r^{\top}x \ \ \ \mathrm{s.t.} \ \ \   Ax = b, \ \ \ x \in X , 
\end{equation}
where $Q \in \mathbb{R}^{n \times n}$ is symmetric but not positive semidefinite matrix, $r \in \mathbb{R}^n$, $A \in \mathbb{R}^{m \times n}$, $b \in \mathbb{R}^m$, and  $X = \{x \in \mathbb{R}^n: l_i \leq x_i \leq u_i, \ i=1,\ldots,n  \}$. Here, $h(x)=\mathcal{I}_X(x)$ denotes the indicator function for $X$.
We evaluate our method on four different problem sizes, denoted by $(n \times m)$: $(50\times 10), (100 \times 10), (500\times 50)$, and $(1000 \times 100)$.  For every instance, we set $l_i=0$ and $u_i=5$ for all $i=1,\ldots,n$. 

We generate data as follows: The matrix $Q$ is generated as $Q=(\Tilde{Q} +\Tilde{Q}^\top)/2$, where the entries of $\Tilde{Q}$ are randomly generated from the standard Gaussian distribution $\mathcal{N}(0,1)$. The entries of  $q$ and $A$ are also generated from the standard Gaussian. Moreover, we set $b=Ax$, where $x$ is randomly drawn from the standard Gaussian. The following MATLAB code is used to generate the data:
\vspace{0.05in}

\begingroup
\renewcommand{\baselinestretch}{1.0}
\begin{adjustwidth}{0.15in}{}
\begin{verbatim}
QP.Q1 = randn(n);
QP.Q = (QP.Q1+QP.Q1')/2;  % matrix Q
QP.r = randn(n,1);        % vector r 
QP.A = randn(m,n);        % linear operator A
xx = randn(n,1);          % a random x
QP.b = QP.A*xx;           % vector b
\end{verbatim}
\end{adjustwidth}
\endgroup
\vspace{0.05in}

\begin{figure}[t!]
\centering
\begin{subfigure}{1\textwidth}
\centering
\includegraphics[scale=0.583]{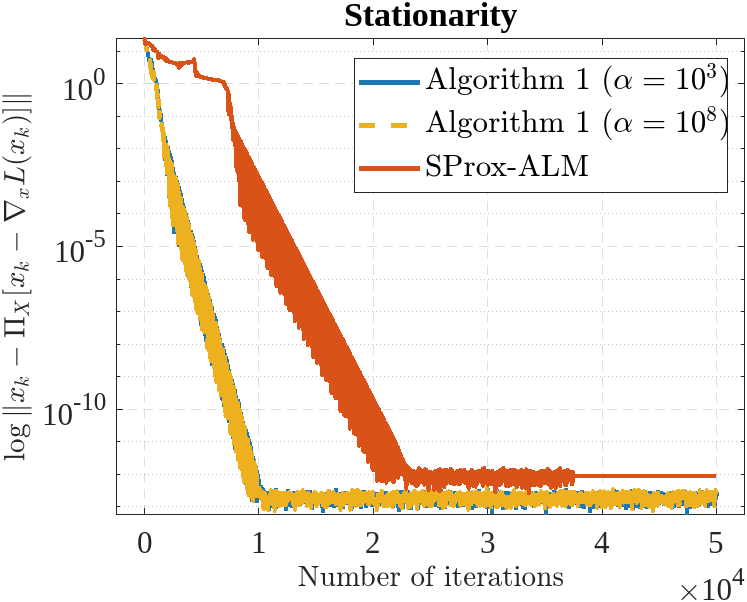} \hspace{0.15in}
\includegraphics[scale=0.583]{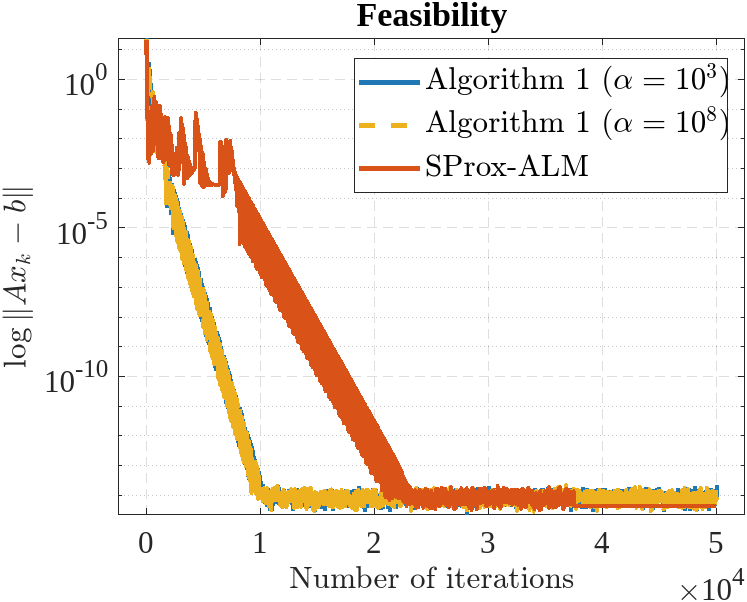}
\caption{{\small An instance with $n=50$ and $m=10$.}}
\end{subfigure}
\vspace{0.01in}

\begin{subfigure}{1\textwidth}
\centering
  \includegraphics[scale=0.583]{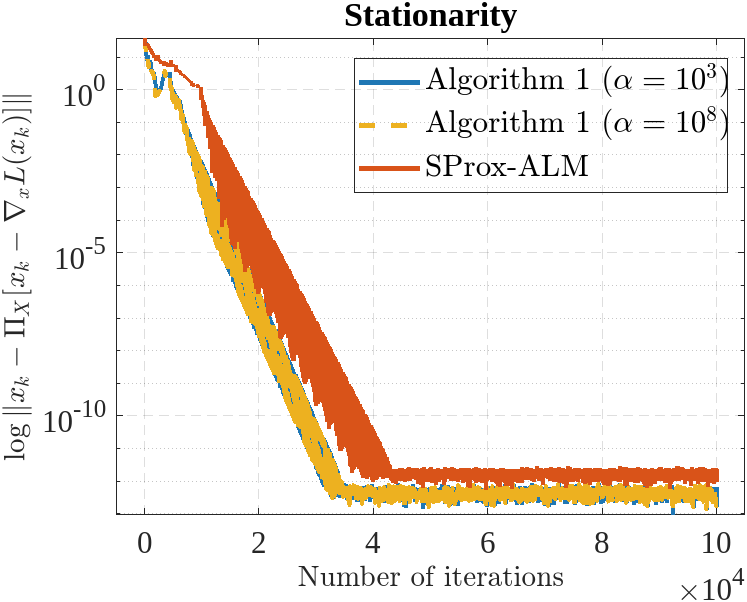}\hspace{0.15in}
  \includegraphics[scale=0.583]{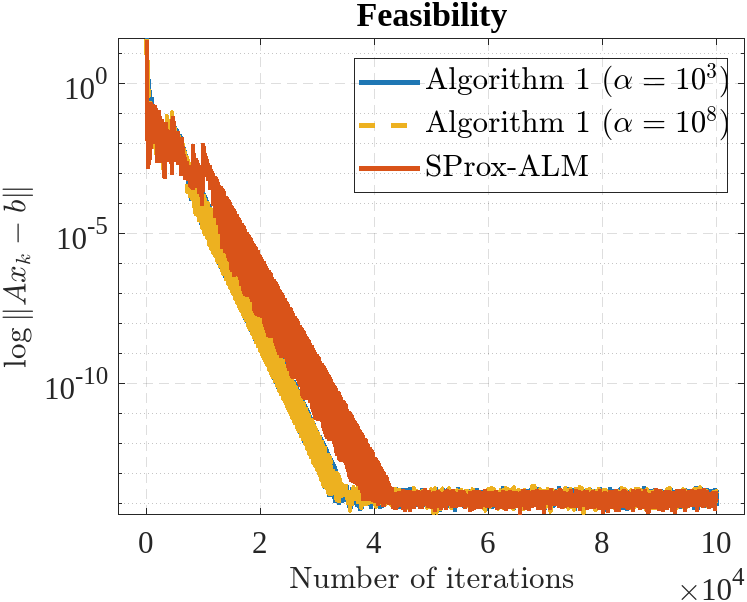}
  \caption{{\small An instance with $n=100$ and $m=10$.}}
\end{subfigure}
\caption{Performance comparison of \textsf{Algorithm 1} with different choices of $\alpha>0$ and \textsf{SProx-ALM} on LCQP \eqref{eq:lcqp}. Here, $\nabla_x L(x_k)$ represents both $\nabla_x \mathcal{L}_{\beta}(x_k,z_k,\lambda_k,\mu_k)$ for \textsf{Algorithm 1} and $\nabla_x L(x_k,\lambda_k)$ for \textsf{SProx-ALM}. Interestingly, the convergence behavior of \textsf{Algorithm 1} remains consistent, regardless of how large the (false) penalty parameter $\alpha$ is, even for a significantly large value $\alpha=10^8$. } \label{fig1}
\end{figure}
\noindent For all experiments, the parameters for \textsf{Algorithm \ref{algorithm1}} are set simply as follows:
$$\beta=0.5, \ \ \alpha = 10^3, \ \ \delta_0=0.5, \ \  r = 1 - 10^{-7}, \ \ \eta = \frac{1}{L_Q +  \left(2 + \frac{1}{1+ \alpha \beta}\right) \rho \sigma_{\max}^2},$$
where $L_Q$ is the eigenvalue of $Q$ with the largest absolute value. The initial point $x_0$ is generated randomly and $(z_0, \lambda_0, \mu_0)=({0,0,0})$ for all test instances.

The \textsf{SProx-ALM} is given by 
\begin{equation}
    \begin{cases}
    \lambda_{k+1} = \lambda_k + \Tilde{\alpha} (A x_k-b); \\
    x_{k+1} = \Pi_X [x_k - c \nabla_x K (x_k,z_k,\lambda_{k+1})]; \\
    z_{k+1} = z_k + \Tilde{\beta} (x_{k+1} - z_k),
    \end{cases} \notag
\end{equation}
where 
$K(x, z,\lambda):= {L}(x,\lambda) + \frac{p}{2}\|x - z \|^2 \  \text{with}  \ {L}(x,\lambda) := f(x) + \left\langle \lambda, Ax - b \right\rangle + \frac{\gamma}{2} \| Ax - b \|^2$, 
and $\Pi_X[\cdot]$ is the projection operator onto $X$. The parameters for SProx-ALM are set as in \cite[Section 6.2]{zhang2020proximal}:
$$\Tilde{\alpha} = \frac{\gamma}{4}, \ \ p = 2L_Q, \ \ \Tilde{\beta} = 0.5, \ \  
c = \frac{1}{2(L_Q + p + \gamma \sigma_{\max}^2 )}.$$  
The initial point $x_0=z_0$ is randomly generated with $\lambda_0 = 0$.

\begin{figure}[t!]
\centering
\begin{subfigure}{1\textwidth}
\centering
\includegraphics[scale=0.583]{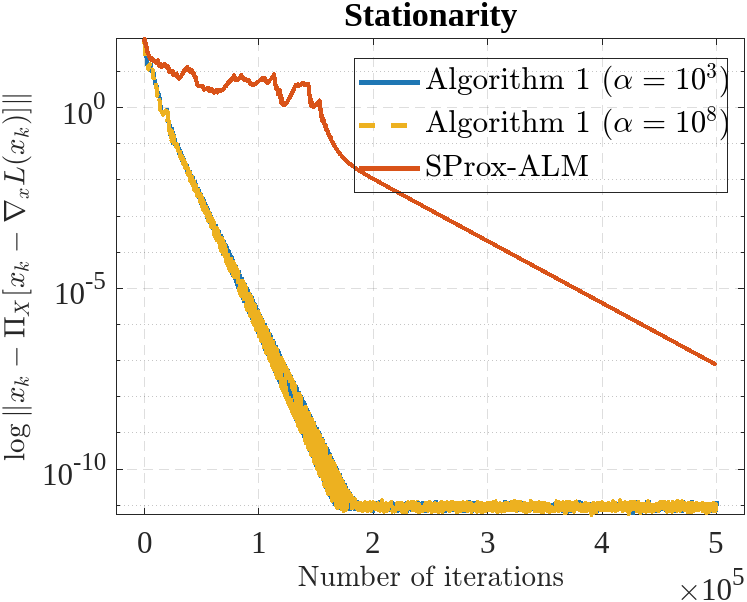} \hspace{0.15in}
\includegraphics[scale=0.583]{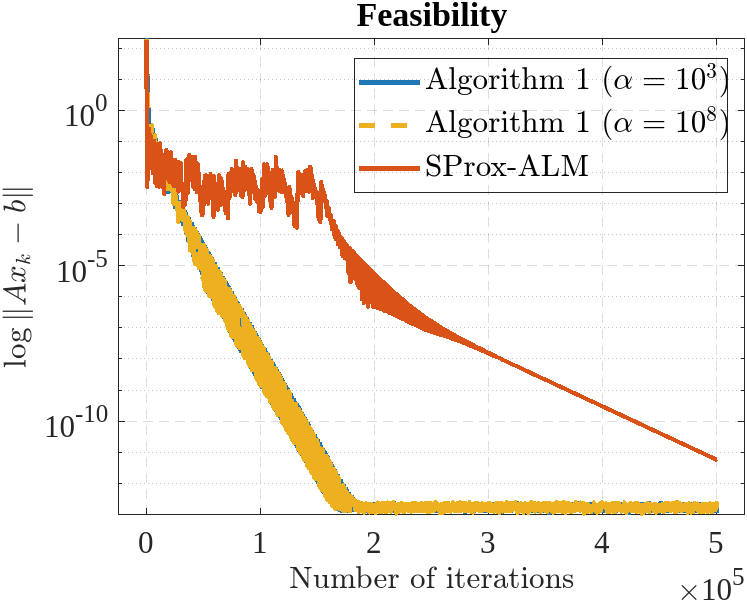}
\caption{{\small An instance with $n=500$ and $m=50$.}}
\end{subfigure}
\vspace{0.01in}

\begin{subfigure}{1\textwidth}
\centering
  \includegraphics[scale=0.583]{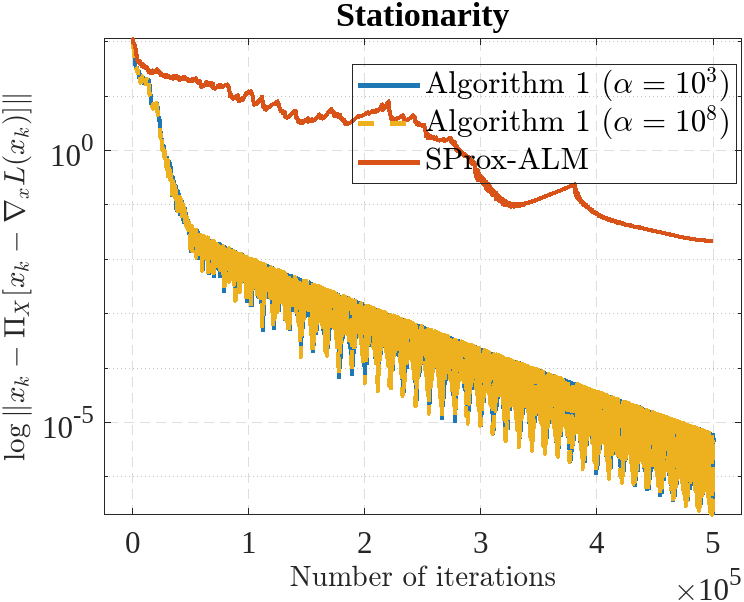}\hspace{0.15in}
  \includegraphics[scale=0.583]{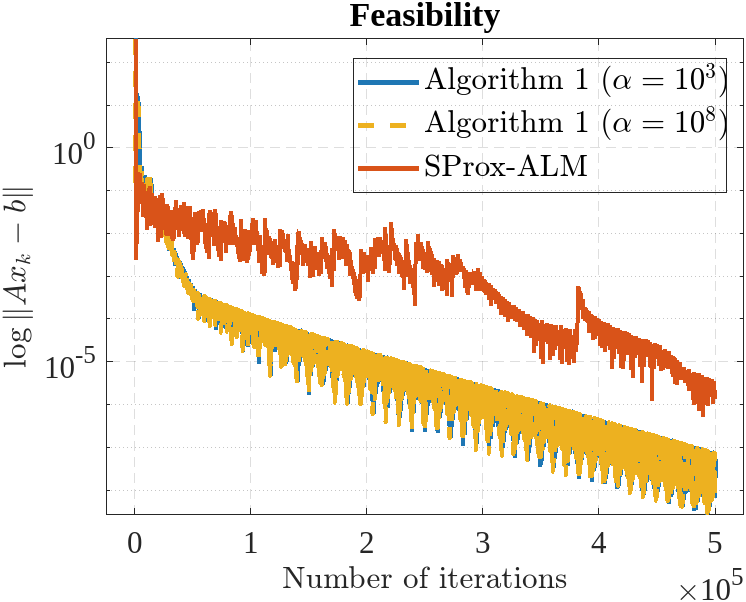}
  \caption{{\small An instance with $n=1000$ and $m=100$.}}
\end{subfigure}
\caption{Performance comparison of \textsf{Algorithm 1} with different choices of $\alpha>0$ and \textsf{SProx-ALM} on larger LCQP \eqref{eq:lcqp} instances.  Regardless of how large value of  $\alpha$, \textsf{Algorithm 1} exhibits consistent behavior in larger settings.} \label{fig2}
\end{figure}

To evaluate the convergence behaviors of 
 {Algorithm \ref{algorithm1}} and {SProx-ALM}, we use the quantities to measure the stationarity (first-order optimality):
\begin{align}
  \left\|x_k - \Pi_X [x_k - \nabla_x \mathcal{L}_{\beta}(x_k,z_k,\lambda_k,\mu_k)] \right\|  \quad \text{and} \quad
 \left\|x_k - \Pi_X [x_k - \nabla_x L(x_k,\lambda_k)] \right\|  \notag
\end{align}
for Algorithm \ref{algorithm1} and SProx ALM, respectively. For the feasibility measure, the quantity $\|Ax_k -b\|$ is used for both algorithms.  

The numerical results clearly demonstrate that Algorithm \ref{algorithm1} effectively and efficiently solves the LCQP instances. Figures \ref{fig1} and \ref{fig2} illustrate the performance of Algorithm \ref{algorithm1} for all four instances, showing its faster convergence compared to SProx-ALM. In particular, Figure \ref{fig2} indicates that Algorithm \ref{algorithm1} significantly outperforms SProx-ALM when applied to larger problems. Furthermore, Figure \ref{fig2} also highlights a practical strength of Algorithm \ref{algorithm1};  it provides a consistent reduction of both the stationarity and feasibility gaps, which aligns with our theoretical findings. We make a few remarks on our numerical experiments:
\vspace{0.05in}

\begin{enumerate}
\item Setting parameters for Algorithm \ref{algorithm1} is straightforward, involving considerations of $L_f$ and $\sigma_{\max}$, along with a simple setting of $\beta = 0.5$. In addition, it is noteworthy that the performance of Algorithm \ref{algorithm1} is not sensitive to the value of (false) penalty parameter $\alpha$. Even with a significantly large $\alpha=10^8$, Algorithm \ref{algorithm1} consistently performs well. This robustness stems from the fact that $\alpha$ primarily influences the $z$-update \eqref{eq:z_update} through exact minimization. On the other hand, we observed that the performance of SProx-ALM is highly sensitive to the choices of parameters such as $\gamma$, $\tilde{\alpha}$, and $\tilde{\beta}$, particularly when dealing with larger instances. 

\item For large problems, it is crucial to choose the reduction ratio $r$ closer to 1, as these problems require more iterations to achieve the desired levels of stationarity and feasibility (Figure \ref{fig2}). A smaller $r$ can cause $\mu_k$ to converges a certain point before $\lambda_k$ satisfies the KKT conditions (Remark \ref{remark_step_mu}). Therefore, for large-scale problems, choosing $r$ closer to 1 is crucial to ensure convergence to an $\epsilon$-KKT solution in our proposed algorithm in practice. 
\end{enumerate}

\section{Conclusions}
This paper presented a novel primal-dual framework that incorporates false penalization and dual smoothing to solve linearly constrained nonconvex optimization problems. Our method achieves the best-known complexity bound of $\mathcal{O}(1/\epsilon^2)$ with theoretical guarantees. The proposed method has distinct advantages in that it does not rely on some restrictive assumptions often imposed by other algorithms and it ensures a consistent reduction in both first-order optimality and feasibility gaps. Experimental results validate that our algorithm performs better than the existing single-loop algorithm. Future research could consider extending this framework to tackle nonlinear and/or stochastic constrained nonconvex optimization problems, which could potentially broaden its applicability across a wide range of domains.

\bibliographystyle{informs2014}		
\bibliography{references}

\end{document}